\documentclass[12pt]{amsart}
\usepackage[all]{xypic}
\usepackage{amsmath}
\usepackage{amscd}
\usepackage{amssymb}
\usepackage{enumerate}
\usepackage{amsfonts}
\usepackage{graphicx}
\usepackage[all]{xy}
\usepackage{mathrsfs}
\usepackage[active]{srcltx}
\usepackage{ftnxtra}
\usepackage{hyperref}
\usepackage{bm}

\DeclareFontEncoding{OT2}{}{}

\newcommand{\M}{\operatorname{M}}

\newcommand{\rar}{\longrightarrow}
\newcommand{\hrar}{\hookrightarrow}

\newcommand{\rarab}[1]{\overset{#1}{\longrightarrow}}

\newcommand{\al}{\alpha}
\newcommand{\be}{\beta}
\newcommand{\ga}{\gamma}
\newcommand{\Ga}{\Gamma}
\newcommand{\de}{\delta}

\newcommand{\la}{\lambda}
\newcommand{\La}{\Lambda}

\newcommand{\ka}{\kappa}
\newcommand{\eps}{\epsilon}
\newcommand{\sg}{\sigma}

\newcommand{\om}{\omega}
\newcommand{\Om}{\Omega}

\newcommand{\Sg}{\Sigma}

\newcommand{\bA}{{\mathbb A}}

\newcommand{\bC}{{\mathbb C}}

\newcommand{\bF}{{\mathbb F}}

\newcommand{\bN}{{\mathbb N}}

\newcommand{\bQ}{{\mathbb Q}}
\newcommand{\bR}{{\mathbb R}}
\newcommand{\bS}{{\mathbb S}}
\newcommand{\bT}{{\mathbb T}}

\newcommand{\bZ}{{\mathbb Z}}

\newcommand{\cA}{{\mathcal A}}

\newcommand{\cF}{{\mathcal F}}

\newcommand{\cO}{{\mathcal O}}
\newcommand{\cP}{{\mathcal P}}

\newcommand{\caR}{{\mathcal R}}
\newcommand{\cS}{{\mathcal S}}

\newcommand{\cU}{{\mathcal U}}

\newcommand{\cX}{{\mathcal X}}

\newcommand{\fa}{{\mathfrak a}}
\newcommand{\fb}{{\mathfrak b}}

\newcommand{\fp}{{\mathfrak p}}
\newcommand{\fq}{{\mathfrak q}}

\newcommand{\abs}[1]{\vert #1\vert}

\newcommand{\End}{\operatorname{End}}

\newcommand{\Hom}{\operatorname{Hom}}

\newcommand{\Aut}{\operatorname{Aut}}
\newcommand{\Span}{\operatorname{Span}}

\newcommand{\id}{\operatorname{id}}

\newcommand{\diag}{\operatorname{diag}}

\newcommand{\Gal}{\operatorname{Gal}}
\newcommand{\GL}{\operatorname{GL}}

\newcommand{\val}{\operatorname{val}}
\newcommand{\rad}{\operatorname{rad}}

\newcommand{\tr}{\operatorname{Tr}}

\newcommand{\sbr}{\smallbreak}

\newcommand{\indlim}[1]{\lim\limits_{\underset{N}{\longrightarrow}}}
\newtheorem{thm}{Theorem}
\newtheorem{cor}[thm]{Corollary}
\newtheorem{lem}[thm]{Lemma}

\newtheorem{prop}[thm]{Proposition}
\newtheorem*{t:mainth0}{Theorem}

\theoremstyle{remark}
\newtheorem{rem}[thm]{Remark}

\newtheorem{defin}[thm]{Definition}

\newtheorem{prob}{Problem}
\newtheorem{question}{Question}
\newtheorem{example}{Example}
\newtheorem{case}{Case}

\newcommand{\bbe}{\begin{equation}}
\newcommand{\ee}{\end{equation}}

\title[Endomorphism rings of toroidal solenoids]{Endomorphism rings of toroidal solenoids}

\author{Maria Sabitova}
\address{Department of Mathematics\\
CUNY Queens College\\ 65-30 Kissena Blvd.\\ Flushing, NY
11367\\USA} 
\email{Maria.Sabitova@qc.cuny.edu}

\date{\today}

\begin{document}

\tableofcontents

\begin{abstract}  
We study the endomorphism ring $\End(G_A)$ of a subgroup $G_A$ of $\bQ^n$ defined by a non-singular 
$n\times n$-matrix $A$ with integer entries. 
In the case when the characteristic polynomial of $A$ is irreducible and an extra assumption holds if $n$ is not prime, we show that $\End(G_A)$ is commutative and can be identified with a subring of the number field generated by an eigenvalue of $A$.
The obtained results can be applied to studying endomorphisms of associated toroidal solenoids and $\bZ^n$-odometers.
In particular, we build a connection between toroidal solenoids and $\cS$-integer dynamical systems, provide a formula for the number of periodic points of a toroidal solenoid endomorphism, and show that the linear representation group of a 
$\bZ^n$-odometer is computable.
 \end{abstract}


\maketitle

\section{Introduction}
We study the endomorphism ring of a subgroup of $\bQ^n$ defined by a matrix with integer entries. The group arises naturally as the character group of a toroidal solenoid. More precisely, let $A\in \operatorname{M}_n(\bZ)$ be a non-singular $n\times n$-matrix with integer entries. Denote
$$
G_A=\left\{A^{k}{\bf x}\, \vert \, {\bf x}\in\bZ^n,\,k\in\bZ\right\},\quad \bZ^n\subseteq G_A\subseteq\bQ^n.
$$
One can readily verify that $G_A$ is a subgroup of $\bQ^n$. In \cite{s1} and \cite{s2}, we study the classification problem of groups $G_A$. In particular, given two matrices $A,B\in \operatorname{M}_n(\bZ)$ with integer entries, we answer the question of when the corresponding groups $G_A$, $G_B$ are isomorphic as abstract groups in terms of the matrices $A,B$. We cover the case $n=2$ in \cite{s1} and the case of an arbitrary $n$ in \cite{s2}. Groups $G_A$ arise in connection with toroidal solenoids. Toroidal solenoids defined by non-singular matrices with integer entries were introduced by M.~C.~McCord in 1965 \cite{m}. 
A toroidal solenoid $\bS_A$ defined by a non-singular $A\in\operatorname{M}_n(\bZ)$ is an  
$n$-dimensional topological abelian group. It is compact, metrizable, and connected, but not locally connected and not path connected. Toroidal solenoids are examples of inverse limit dynamical systems. When $n=1$ and $A=d$, $d\in\bZ$, solenoids  are called $d$-{\it adic solenoids} or {\it Vietoris solenoids}. The first examples were studied by L. Vietoris in 1927 for $d=2$
\cite{v} and later in 1930 by van Dantzig for an arbitrary $d$ \cite{d}. It is known that the first \^Cech cohomology group $H^1(\bS_A,\bZ)$ of $\bS_A$ is isomorphic to $G_{A^t}$, where  $A^t$ is the transpose of $A$.
 On the other hand, since $\bS_A$ is a compact connected abelian group, $H^1(\bS_A,\bZ)$ is isomorphic to the character group 
 $\widehat{\bS_A}$ of $\bS_A$. Thus $\widehat{\bS_A}\cong G_{A^t}$ and, using Pontryagin duality theorem, 
 $\bS_A\cong \widehat{G_{A^t}}$ as topological groups, where $G_{A^t}$ is endowed with the discrete topology, 
the dual $\widehat{G_{A^t}}$ is endowed with the compact-open topology, and $\bS_A$ is endowed with the  topology of an inverse limit. Groups $G_A$ also arise as the first cohomology groups of constant base $\bZ^n$-odometers 
$X_{A^t}$
defined by $A^t$ \cite{gps} and as the dimension groups of subshifts of finite type defined by $A^t$ \cite{bs}. 

\sbr

In this paper, we study the endomorphism ring $\End(G_A)$ of $G_A$ for an arbitrary $n$. Based on our work in \cite{s1} and \cite{s2}, we give a general criterion for a matrix with rational entries $T\in\M_n(\bQ)$ to define an endomorphism of $G_A$, equivalently, $T(G_A)\subseteq G_A$ (Theorem \ref{th:prelim0}) as well as a more concrete description when $n=2$ (Section \ref{ss:two}). In the case when the characteristic polynomial of $A$ is irreducible and an extra assumption holds if $n$ is not prime, we prove that $T\in\End(G_A)$ is either zero or $T$ commutes with $A$. Moreover,  the eigenvalues of $T$ are elements of the form $a\la^k$, where $\la$ is an eigenvalue of $A$, $k\in\bZ$,  and $a$ is an algebraic integer of the number field $\bQ(\la)$ generated by $\la$ (Proposition \ref{pr:end}). This implies that $\End(G_A)$ is a commutative ring, and $\Aut(G_A)$ is a finitely generated abelian group. As a consequence, considering a toroidal solenoid $\bS_A$ as a dynamical system with the automorphism $\sg$ defined by multiplication by $A$, it shows that every homomorphism of $\bS_A$ as a topological group is a morphism of the dynamical system 
$(\bS_A,\sg)$. Other noteworthy consequences include the connection between toroidal solenoids and $\cS$-integer dynamical systems defined in \cite{cew}. It turns out that one can consider a toroidal solenoid as a similar but more general object than an $\cS$-integer dynamical system. The connection allows us to obtain a formula for the number of periodic points of an endomorphism of $\bS_A$ similar to the one in \cite{cew}. The case of $n=2$ is covered in \cite{hl} for a more general class of toroidal solenoids, and our formula holds in higher-dimensions for $\bS_A$. We also apply our results to $\bZ^n$-odometers. We recover results of \cite{cp} in the case when $n=2$ and generalize them to higher-dimensions. We also discuss the question of \cite{cp} on the computability of the linear representation group of a $\bZ^n$-odometer defined by an integer matrix. 

\section{Notation}
\noindent $A,B\in\M_n(\bZ)$ non-singular \\
$h_A\in\bZ[x]$ characteristic polynomial of $A$ \\
$G_A=\left\{A^{k}{\bf x}\, \vert \, {\bf x}\in\bZ^n,\,k\in\bZ\right\}$ \\
$\caR=\bZ\left[\frac{1}{\det A} \right]$ \\
$\cP=\cP(A)=\{\text{primes }p\in\bN\text{ dividing }\det A\}$ \\
$\cP'=\cP'(A)=\left\{p\in\cP\,\vert\,\,h_A\not\equiv x^n\,(\text{mod }p)\right\}$ \\
$t_p=$ multiplicity of zero in the reduction of ${h}_A$ modulo $p$ \\
$\bQ_p=$ field of $p$-adic numbers \\
$\bZ_p=$ ring of $p$-adic integers \\
$\bF_p=$ finite field with $p$ elements \\
$\bQ_p^n=(\bQ_p)^n=\bQ_p\times\cdots\times\bQ_p$ \\
$\bZ_p^n=(\bZ_p)^n=\bZ_p\times\cdots\times\bZ_p$ \\
$\overline{G}_{A,p}=G_A\otimes_{\bZ}\bZ_p$ \\
$\overline{\bQ}=$ algebraic closure of $\bQ$ \\
$\la=$ eigenvalue of $A$ \\
$K=\bQ(\la)$ \\ 
${\bf u}=\left(\begin{matrix} u_1 & \ldots & u_n \end{matrix}\right)^t$ eigenvector of $A$ corresponding to $\la$ \\
$\bZ[{\bf u}]=\{m_1u_1+\cdots +m_nu_n\,\vert\,m_1,\ldots,m_n\in\bZ \}$ \\
$Y_{A}({\bf u},\la)=\{m_1\la^{k_1}u_1+\cdots +m_n\la^{k_n}u_n\,\vert\,m_1,\ldots,m_n,k_1,\ldots,k_n\in\bZ \}$ \\
$\{\la_1,\ldots,\la_n\}=$ eigenvalues of $A$ \\
$\{\sg_1=\id,\sg_2,\ldots,\sg_n\}=$ embeddings of $K$ into $\overline{\bQ}$ \\
$M=\left(\begin{matrix} \sg_1({\bf u}) & \ldots & \sg_n({\bf u}) \end{matrix}\right)\in\M_n(\overline{\bQ})$ \\
$m=(\det M)^2\in\bZ$ \\
$\cO_K=$ ring of integers of $K$ \\
$\cO_K^{\times}=$   units of $\cO_K$ \\
$\fp=$ prime ideal of $\cO_K$ above $p$ \\
$\val_{\fp}(x)=$ $\fp$-adic valuation of $x\in K$ \\
$\cS=$ a set of prime ideals of $\cO_K$ \\
$\cO_{K,\cS}=\{x\in K\,\vert\, \val_{\fp}(x)\geq 0\text{ for any prime ideal }\fp\text{ of }\cO_K\text{ not in }\cS\}$ \\ 
$\cU_{K,\cS}=\{x\in K\,\vert\, \val_{\fp}(x)=0\text{ for any prime ideal }\fp\text{ of }\cO_K\text{ not in }\cS\}$ \\
$\cS_{\la}=$ all prime ideals of $\cO_K$ dividing $\la$ \\
$K_{\fp}=$ completion of $K$ with respect to $\fp$ \\
$\cO_{\fp}=$ ring of integers of $K_{\fp}$ \\
$\cX_{A,\fp}=$ $\Span_K\{\text{generalized }\la\text{-eigenvectors of }A, \fp\,\vert\,\la\}$ \\
$\rad(n)=$ product of all distinct prime divisors of $n\in\bZ$ \\
$(u,v)=$ greatest common divisor of $u,v\in\bZ$ \\
$\bS_A=$ toroidal solenoid defined by $A$ \eqref{eq:solen} \\ 
$\widehat{G}=\Hom_{cont}(G,\bT^1)$  Pontryagin dual of a topological group $G$ \\
$X_A=$ $\bZ^n$-odometer defined by $A$ \eqref{eq:xg} \\
$\vec{N}(X_A)=$ linear representation group of $X_A$ \eqref{eq:linrepgr} \\

\section{Endomorphisms of $G_A$}
This section presents key results on endomorphisms of $G_A$, derived as consequences from the proofs in \cite{s2} that characterize isomorphisms between two groups of the form $G_A$, $G_B$ ($B\in\M_n(\bZ)$ is non-singular). 

\sbr

\subsection{Localization and easy cases}\label{ss:loc}
Let $A\in \operatorname{M}_n(\bZ)$ be a non-singular $n\times n$-matrix with integer entries. Denote
\bbe\label{eq:g}
G_A=\left\{A^{k}{\bf x}\, \vert \, {\bf x}\in\bZ^n,\,k\in\bZ\right\},\quad \bZ^n\subseteq G_A\subseteq\bQ^n,
\ee
where ${\bf x}\in\bZ^n$ is written as a column. Denote by $\End(G_A)$ 
the endomorphism ring of $G_A$, consisting of all (group) homomorphisms $\phi:G_A
\rar G_A$.  
Denote 
\bbe\label{eq:car}
\caR=\caR(A)=\bZ\left[\frac{1}{\det A} \right]=\left\{\frac{k}{(\det A)^l}\,\,\Big\vert\,\,k,l\in\bZ  \right\}.
\ee
If $T\in\End(G_{A})$, then $T\in\M_n(\caR)$. Indeed, one can check that any homomorphism from $G_A$ to $G_A$ is given by a matrix $T\in\M_n(\bQ)$. Moreover, from the definition of $G_A$, there exists $i\in\bN\cup\{0\}$ such that $A^iT\in\M_n(\bZ)$, hence $T\in\M_n(\caR)$. Thus, $\End(G_A)\subseteq \M_n(\caR)$. Note that 
$\bZ[A,A^{-1}]\subseteq \End(G_A)$. 
Let 
$$
\cP=\cP(A)=\{\text{primes }p\in\bN\text{ dividing }\det A\},
$$
\begin{eqnarray*}
\cP'=\cP'(A)&=&\left\{p\in\cP\,\vert\,h_A\not\equiv x^n\,(\text{mod }p)\right\},
\end{eqnarray*}
where $h_A\in\bZ[x]$ is the characteristic polynomial of $A$. 

\sbr

For a prime $p\in\bN$, let $\bZ_p$ denote the ring of 
$p$-adic integers and let $\bQ_p$ denote the field of $p$-adic numbers. Let $\overline{G}_{A,p}=G_A\otimes_\bZ{\bZ_p}$, so that
$$
\overline{G}_{A,p}=\left\{A^{k}{\bf x}\, \vert \, {\bf x}\in\bZ_p^n,\,k\in\bZ\right\},\quad \bZ_p^n\subseteq \overline{G}_{A,p}\subseteq\bQ_p^n.
$$
We know that 
\bbe\label{eq:zp}
\overline{G}_{A,p}\cong \bQ_p^{t_p}\oplus \bZ_p^{n-t_p}
\ee
as $\bZ_p$-modules, and $t_p$ equals the multiplicity of zero in the reduction of $h_A$ modulo $p$, $0\leq t_p\leq n$ \cite[p. 196, Prop. 3.8]{s1}. 

\sbr

Clearly, for $T\in\M_n(\bQ)$, if 
$T(G_A)\subseteq G_A$, then $T(\overline{G}_{A,p})\subseteq \overline{G}_{A,p}$, {\it i.e.}, every $T\in\End(G_A)$ induces an endomorphism of $\overline{G}_{A,p}$. It turns out that the converse is also true. 

\begin{lem}\label{l:simple}
For $T\in\M_n(\bQ)$, we have that $T\in\End(G_A)$ if and only if $T\in\M_n(\caR)$ and 
$T\in\End(\overline{G}_{A,p})$
for any $p\in\cP'$.
\end{lem}
\begin{proof}
Let $T\in\M_n(\bQ)$. 
By above, the conditions are necessary. We now show that they are sufficient. 
It follows from \cite[p. 183, Lemma 93.2]{f2} that
\bbe\label{eq:fuks}
G_A=\bigcap_{p\in\cP}(\caR^n\cap\overline{G}_{A,p})
\ee
(see also \cite[p. 8, Corollary 2.4]{s2} for more detail). Also, $\overline{G}_{A,p}=\bQ_p^n$ for any prime $p\in\cP\backslash\cP'$ by \eqref{eq:zp}.  Thus, $T(\overline{G}_{A,p})\subseteq\overline{G}_{A,p}$ for any $p\in\cP\backslash\cP'$. Therefore, if $T\in\M_n(\caR)$ and 
$T\in\End(\overline{G}_{A,p})$
for any $p\in\cP'$, then $T\in\End(G_A)$ by \eqref{eq:fuks}.
\end{proof}

\begin{lem}\label{l:easy}
\begin{enumerate}[$(1)$]
\item If $\cP=\emptyset$,  equivalently, $A\in\GL_n(\bZ)$, then 
$$\End({G}_A)=\M_n(\bZ).$$

\item If $\cP'=\emptyset$ and $A\not\in\GL_n(\bZ)$, then $$\End({G}_A)=\M_n(\caR).$$ 
\end{enumerate}
\end{lem}
\begin{proof}
Clearly, $\cP=\emptyset$ if and only if $\det A=\pm 1$ if and only if $G_A=\bZ^n$, and Lemma \ref{l:easy} (1) is clear. 
Lemma \ref{l:easy} (2) follows from Lemma \ref{l:simple}.
\end{proof}

Thus, by Lemma \ref{l:easy}, for the rest of the paper we assume 
 $A\not\in\GL_n(\bZ)$ and $\cP'\ne\emptyset$. 

\subsection{Eigenvectors}\label{ss:eigen} In practice, to apply Lemma \ref{l:simple},  one needs a basis for the decomposition \eqref{eq:zp}. In \cite{s2}, we show that a divisible part of $\overline{G}_{A,p}$ (isomorphic to $\bQ_p^{t_p}$)
can be described by generalized eigenvectors of $A$ and to treat a reduced part of $\overline{G}_{A,p}$ (isomorphic to $\bZ_p^{n-t_p}$) we need a characteristic of $G_A$. For the reader's convenience,  we recall those results 
 and put them together in a criterion for $T\in\M_n(\bQ)$ to be an endomorphism of $G_A$ (Theorem \ref{th:prelim0} below). We first introduce notation. 

\sbr

Throughout the text, $\overline{\bQ}$ denotes a fixed algebraic closure of $\bQ$. 
Let $F\subset \overline{\bQ}$ be a finite extension of $\bQ$ that contains all the eigenvalues of $A$. Let $\cO_F$ denote the ring of integers of $F$.  Throughout the paper,  
$\la_1,\ldots,\la_n\in\cO_F$ denote (not necessarily distinct) eigenvalues of $A$ and 
$\{{\bf u}_1,\ldots,{\bf u}_n\}$ denotes a Jordan canonical basis of $A$. 
Without loss of generality, we can assume that each
${\bf u}_i\in(\cO_F)^n$, $i=1,\ldots,n$. 
 For a prime $p\in\bN$ let $\fp$ be a prime ideal of $\cO_F$ above $p$ and let $\cX_{A,\fp}$ 
denote the span over $F$ of vectors in $\{{\bf u}_1,\ldots,{\bf u}_n\}$ 
corresponding to eigenvalues divisible by $\fp$. 
Note that 
$$
\dim_F \cX_{A,\fp}=t_p,
$$
where $t_p=t_p(A)$ denotes the multiplicity of zero in the reduction 
$\bar{h}_A$ modulo $p$ of the characteristic polynomial $h_A$ of $A$, $0\leq t_p\leq n$. Indeed, 
$\dim_F \cX_{A,\fp}$ is the number of 
eigenvalues (with multiplicities) of $A$ divisible by $\fp$. One can write $h_A=(x-\la_1)\cdots (x-\la_n)$ over 
$\cO_F$. Considering the reduction $\bar{h}_A$ of $h_A$ modulo $\fp$, we see that the number of eigenvalues of $A$ divisible by $\fp$ is equal to the multiplicity 
$t_p$ of zero in $\bar{h}_A$. Equivalently, $\cX_{A,\fp}$ is generated over $F$ by generalized 
$\la$-eigenvectors of $A$ for any eigenvalue $\la$ of $A$ divisible by $\fp$.

\sbr

 Let $p\in\bN$ be a prime and let $a\in\bZ_p$, $a=\sum_{i=0}^{\infty}a_ip^i$, $\forall a_i\in\{0,1,\ldots,p-1\}$. For $k\geq 1$, we denote  
$$
a^{(k)}=a_0+a_1p+\cdots +a_{k-1}p^{k-1}\in\bZ.
$$ 
Similarly, for ${\bf x}=\left(\begin{matrix} x_1 & \ldots & x_n \end{matrix}\right)\in\bZ_p^n$, 
$x_1,\ldots, x_n\in\bZ_p$, and $k\geq 1$, we denote
$$
{\bf x}^{(k)}=\left(\begin{matrix} x_1^{(k)} & \ldots & x_n^{(k)} \end{matrix}\right)\in\bZ^n.
$$
Finally, for ${\bf x}\in\bZ_p^n$,  we denote  
$$
p^{-\infty}{\bf x}=\{p^{-k}{\bf x}^{(k)}\,\vert\, \,k\in\bN  \}\subset\bQ^n.
$$


\begin{lem}\label{lem:gen-n}
There exists a basis 
$\{{\bf f}_1,\ldots, {\bf f}_n\}$ of $\bZ^n$ such that for any $p\in\cP'$ there are $\al_{pij}\in\bZ_{p}$, 
$i\in\{1,\ldots,t_p\}$, $j\in\{t_p+1,\ldots,n\}$, and $G_A$ is generated over $\bZ$ by
$$
\{{\bf f}_1,\ldots,{\bf f}_n, q^{-\infty}{\bf f}_1, \ldots, q^{-\infty}{\bf f}_n, 
p^{-\infty}{\bf x}_{pi}\,\,\vert\,\, p\in\cP',\,\, q\in\cP\backslash\cP',\,\, 1\leq i\leq t_p\},
$$
where 
$$
 {\bf x}_{pi}={\bf f}_i+\sum_{j=t_p+1}^n\al_{pij}{\bf f}_j. 
$$
\end{lem}

\begin{defin}
Let $\{{\bf f}_1,\ldots, {\bf f}_n\}$ and $\al_{pij}\in\bZ_p$ be as in  Lemma \ref{lem:gen-n}.
The set
$$
M(A;{\bf f}_1,\ldots, {\bf f}_n)=\{\al_{pij}\in\bZ_p\,\vert\,p\in\cP',\,1\leq i\leq t_p<j\leq n\}
$$
is called a {\it characteristic} of $G_A$ relative to the ordered basis $\{{\bf f}_1,\ldots, {\bf f}_n\}$ \cite{gm}. 
\end{defin}

By conjugating $A$ by a matrix $S\in\GL_n(\bZ)$, without loss of generality, we can assume that we have a characteristic of $G_A$ relative to the standard basis $\{{\bf e}_1,\ldots, {\bf e}_n\}$ 
\cite[Lemma 3.8]{s2}.
The next theorem gives a necessary and sufficient criterion for $T\in\M_n(\bQ)$ to be an endomorphism of $G_A$, equivalently, $T(G_A)\subseteq G_A$. The proof follows easily from the proof of 
\cite[Theorem 4.3]{s2}, which gives a criterion for when $T(G_A)=G_B$ for a non-singular $B\in\M_n(\bZ)$. 
Theorem \ref{th:prelim0} works well in practice, since there is an algorithm to produce a characteristic of $G_A$ out of generalized eigenvectors of $A$ (see \cite[Remark 4.5]{s2} and \cite{gm}).

\begin{thm}\label{th:prelim0}
Let $A\in\M_n(\bZ)$ be non-singular, $\cP'\ne\emptyset$, let $F\subset \overline{\bQ}$ be any finite  extension of $\bQ$ that contains all the eigenvalues of $A$, and assume $G_A$ has a characteristic 
$$
M(A;{\bf e}_1,\ldots, {\bf e}_n)=\{\al_{pij}\in\bZ_p\,\vert\,p\in\cP',\,1\leq i\leq t_p(A)<j\leq n\}. 
$$
For $T\in\M_n(\bQ)$, we have that $T(G_A)\subseteq G_A$ if and only if 
 $T\in\M_n(\caR)$, 
for any $p\in\cP'$ and a prime ideal $\fp$ of $\cO_F$ above $p$
we have that
\bbe\label{eq:class0}
 T(\cX_{A,\fp})\subseteq \cX_{A,\fp},
\ee
and any $j$-th column $\left(\begin{matrix} \ga_{1j} & \ldots & \ga_{nj} \end{matrix}\right)$ of $T$, 
$j\in\{t_p+1,\ldots,n\}$, satisfies
\bbe\label{eq:apcond}
\ga_{kj}-\sum_{i=1}^{t_p}\ga_{ij}\al_{pik}\in\bZ_p  \text{ for any } k\in\{t_p+1,\ldots,n\}.
\ee
\end{thm}
\subsection{$2$-dimensional case}\label{ss:two}
For $n\in\bZ$, $n\ne\pm1 $, let $\rad(n)\in\bN$ be the product of all distinct prime divisors $p\in\bN$ of $n$.

\sbr 

If $n=2$, then there are three cases distinguished in \cite{s1}: \\

{\bf (a)} the characteristic polynomial $h_A\in\bZ[x]$ of $A$ is irreducible (equivalently, $A$ has no rational eigenvalues), \\

{\bf (b)} $h_A$ is reducible  (equivalently, $A$ has eigenvalues $\la_1,\la_2\in\bZ$), $\rad(\la_1)$ does not divide $\rad(\la_2)$, and $\rad(\la_2)$ does not divide $\rad(\la_1)$, \\

{\bf (c)} $h_A$ is reducible and every prime dividing one eigenvalue divides the other, {\it e.g.}, $\rad(\la_2)$ divides $\rad(\la_1)$ (denoted by 
$\rad(\la_2)\,\vert \rad(\la_1)$). \\

{\bf Case (a)} is treated in Section \ref{ss:22irr} below. \\ 

{\bf Case (b)}.
Note that if $n=2$ and $\cP'\ne\emptyset$, then $\det A\ne\pm 1$, $A$ has distinct eigenvalues $\la_1,\la_2\in\bZ$, and hence $A$ is diagonalizable over $\bQ$. Moreover, there exists 
$S\in\operatorname{GL}_2(\bZ)$ such that $SAS^{-1}=M\La M^{-1}$, where
\bbe\label{eq:dim2}
\La=\left(\begin{matrix}
\la_1 & 0 \\
0 & \la_2 
\end{matrix}
\right), \quad 
M=\left(\begin{matrix}
1 & u \\
0 & v 
\end{matrix}
\right),
\,\, \la_1,\la_2,u,v\in \bZ,\,\,(u,v)=1,\,\,v\,\vert\,(\la_1- \la_2),
\ee
where $(u,v)=1$ means that $u,v$ are coprime 
\cite[Corollary A.2]{s1}. Since $S(G_{A})=G_{SAS^{-1}}$, {\it i.e.}, 
$G_A$, $G_{SAS^{-1}}$ are isomorphic, without loss of generality, we can assume that $A$ itself is upper-triangular and has the form $A=M\La M^{-1}$.
\begin{thm}\label{th:2dimredb}
Assume $A=M\La M^{-1}$, where $M,\La$ are given by $\eqref{eq:dim2}$, and $\cP'\ne\emptyset$. Assume case $(b)$, i.e., $\rad(\la_1)$ does not divide $\rad(\la_2)$, and $\rad(\la_2)$ does not divide $\rad(\la_1)$. Then $T\in\End(G_A)$ if and only if
\bbe\label{eq:T}
T=MXM^{-1},\quad X=\diag\left(\begin{matrix} x_1 & x_2\end{matrix}\right),
\ee
where $x_i\in\bZ[\la_i^{-1}]$, $i=1,2$, and $\frac{x_1-x_2}{v}\in\caR$.
In particular, $\End(G_A)$ is commutative, isomorphic to a subring of $\bZ[\la_1^{-1}]\times \bZ[\la_2^{-1}]$, and lies inside the centralizer of $A$ in $\M_2(\caR)$.
\end{thm}
\begin{proof}
Assume $T\in\End(G_A)$. Note that $t_p=1$ for any $p\in\cP'$. Thus, in the notation of Section \ref{ss:eigen}, 
$F=\bQ$, $\fp=p$, and 
$\cX_{A,\fp}$ is a one-dimensional vector space over $\bQ$ generated by an eigenvector ${\bf u}$ of $A$ corresponding to an eigenvalue $\la$. Thus,  
\eqref{eq:class0} states that $T({\bf u})=x{\bf u}$, for some $x\in\bQ$. 
In case (b), there exists a prime $p\in\bN$ dividing $\la_2$ that does not divide $\la_1$ and there exists a prime $q\in\bN$ dividing 
$\la_1$ that does not divide $\la_2$, {\it i.e.}, $p,q\in\cP'$. Applying \eqref{eq:class0} to $F=\bQ$, $\fp=p$ and $\fp=q$, we
get that 
$T({\bf u}_1)=x_1{\bf u}_1$, $T({\bf u}_2)=x_2{\bf u}_2$ for eigenvectors ${\bf u}_1,{\bf u}_2$ of $A$ corresponding to
$\la_1,\la_2$, respectively, and $x_1,x_2\in\bQ$. Hence, 
$T=MXM^{-1}$, where $X=\diag\left(\begin{matrix} x_1 & x_2\end{matrix}\right)\in\M_2(\bQ)$. We will use Lemma \ref{l:simple} to show that the remaining conditions hold. One can easily check that for $T$ and $M$ given by \eqref{eq:T} and
\eqref{eq:dim2}, respectively, we have that $T\in\M_2(\caR)$ if and only if $x_1,x_2,\frac{x_1-x_2}{v} \in\caR$.  
Also, 
$T\in\End(\overline{G}_{A,p})$ if and only if for any $m\in\bN\cup\{0\}$ there exists $k_m\in\bN\cup\{0\}$ with
\bbe\label{eq:attm0}
A^{k_m}TA^{-m}\in\M_2(\bZ_p).
\ee
Here, 
$$
A^{k_m}TA^{-m}=M\left(\begin{matrix}
x_1\la_1^{k_m-m} & 0 \\
0 &  x_2\la_2^{k_m-m} 
\end{matrix}
\right)M^{-1}.
$$
Note that any $p\in\cP'$  does not divide $\la_1-\la_2$ and, therefore, $M\in\GL_2(\bZ_p)$. Thus,
\eqref{eq:attm0} holds if and only if $x_1\la_1^{k_m-m},x_2\la_2^{k_m-m}\in\bZ_p$, $p\in\cP'$, which, together with
$x_1,x_2\in\caR$, implies that $x_i\in\bZ[\la_i^{-1}]$, $i=1,2$. Similarly, one shows that \eqref{eq:T} is sufficient for $T\in\End(G_A)$. 
\end{proof}

{\bf Case (c)} is different from cases (a) and (b) in the sense that for $T\in\End(G_A)$, we have that $T({\bf u})$
 is {\it not} an eigenvector of $A$ for 
{\it every} eigenvector ${\bf u}$ of $A$. Namely, there exists $p\in\cP'$ dividing $\la_1$ and hence 
\eqref{eq:class0} applied to $F=\bQ$, $\fp=p$ states that 
$T({\bf u}_1)=x_1{\bf u}_1$ for an eigenvector ${\bf u}_1$ of $A$ corresponding to $\la_1$. However, 
$T({\bf u}_2)$ is not necessarily a multiple of ${\bf u}_2$ for an eigenvector ${\bf u}_2$ of $A$ corresponding to $\la_2$.

\begin{thm}\label{th:2dimredc}
Assume $A=M\La M^{-1}$, where $M,\La$ are given by $\eqref{eq:dim2}$, and $\cP'\ne\emptyset$. Assume case $(c)$, i.e., 
$\rad(\la_2)\,\vert \rad(\la_1)$. Then $T\in\End(G_A)$ if and only if 
\bbe\label{eq:Ttriang}
T=\left(\begin{matrix}
x & y \\
0 &  z 
\end{matrix}
\right)\in\M_2(\caR),\quad z\in\bZ[\la_2^{-1}].
\ee
In particular, $\End(G_A)$ is not commutative and does not lie inside the centralizer of $A$ in $\M_2(\caR)$. 
\end{thm}

\begin{proof} Let $T\in\End(G_A)$. Assume $\rad(\la_2)\,\vert \rad(\la_1)$. Note that any $p\in\cP'$   
divides $\la_1$ and does not divide $\la_2$. Then  
\eqref{eq:class0} applied to $F=\bQ$, $\fp=p$ states that 
$T({\bf e}_1)=x_1{\bf e}_1$ for some $x_1\in\bQ$, since $A$ is upper-triangular and ${\bf e}_1$ is an eigenvector of $A$ corresponding to $\la_1$  (by assumption). 
Therefore, $T$ is also upper-triangular. Let
\bbe\label{eq:uptriang}
T=\left(\begin{matrix}
x & y \\
0 &  z 
\end{matrix}
\right).
\ee
As in the proof of Theorem \ref{th:2dimredb}, $T\in\M_2(\caR)$ and \eqref{eq:attm0} holds for  $T$ given by 
\eqref{eq:uptriang} and any $p\in\cP'$. Taking into account that $\la_2$  is a unit in $\bZ_p$ and $M\in\GL_2(\bZ_p)$, 
this implies $z\in \bZ[\la_2^{-1}]$. Similarly, one shows that \eqref{eq:Ttriang} is sufficient for $T\in\End(G_A)$.  
\end{proof}

\section{Irreducible characteristic polynomial}\label{s:irredcharpoly}
For an eigenvalue $\la\in\overline{\bQ}$ of $A$ 
let $K=\bQ(\la)$ and let $\cS_{\la}$ consist of all prime ideals of the ring of integers
$\cO_K$ of $K$ dividing $\la$. (Note that $\la\in\cO_K$.) We denote by $\cO_{K,{\la}}$ the ring of $\cS_{\la}$-integers, {\it i.e.},
\bbe\label{eq:Sint}
\cO_{K,{\la}}=\{x\in K\,\vert\, \val_{\fp}(x)\geq 0\text{ for any prime ideal }\fp\text{ of }\cO_K\text{ not in }\cS_{\la}\}
=\cO_{K}\left[{\la}^{-1}\right]
\ee 
and
$$
\cU_{K,\cS_{\la}}=\{x\in K\,\vert\, \val_{\fp}(x)=0\text{ for any prime ideal }\fp\text{ of }\cO_K\text{ not in }\cS_{\la}\}
$$
is the group of $\cS_{\la}$-units. In particular, $\caR=\cO_{\bQ,\cP}$.  

\sbr

In the next proposition, we consider the generic case when the characteristic polynomial $h_A\in\bZ[x]$ of $A$ is irreducible. 
We also add an extra assumption that there exists a prime $p\in\bN$ such that 
$n$ and $t_p$ are coprime, denoted by $(n,t_p)=1$. It turns out that if $(n,t_p)=1$, then 
$T({\bf u})$ is a multiple of ${\bf u}$ for any $T\in\End(G_A)$ and any eigenvector ${\bf u}$ of $A$. 
In particular, $\End(G_A)$ is commutative. If $(n,t_p)\ne 1$ for any $p\in\cP'$, then this is not necessarily true (see Example \ref{ex:4} below with a non-commutative $\End(G_A)$).

\begin{prop}\label{pr:end} 
Assume $A\in\M_n(\bZ)$ is non-singular with an irreducible characteristic polynomial 
$h_A\in\bZ[x]$ and $\cP'\ne\emptyset$. 
Assume, in addition, that there exists a prime $p\in\bN$ with 
$(n,t_p)=1$. Then there is a ring embedding $\imath=\imath(A,\la):\End(G_A)\hrar\cO_{K,{\la}}$, which induces  a group embedding 
$\imath:\Aut(G_A)\hrar\cU_{K,\cS_{\la}}$.
\end{prop}
\begin{proof}
Let $T\in\End(G_A)$ be arbitrary. It follows from \cite{s2} that either $T=0$ or $T$ is non-singular and preserves eigenspaces, {\it i.e.}, 
$T{\bf u}, {\bf u}\in K^n$ are both eigenvectors of $A$ corresponding to the same eigenvalue $\la$. 
We provide a comprehensive overview of the argument, both for the sake of completeness and because the results
in  \cite{s2} are presented for isomorphisms from $G_A$ to $G_B$, rather than endomorphisms. Here, $B\in\M_n(\bZ)$ is another non-singular matrix. Nevertheless, the same principles apply.
Let $p\in\cP'$. By \eqref{eq:zp}, $\overline{G}_{A,p}\cong \bQ_p^{t_p}\oplus\bZ_p^{n-t_p}$ as $\bZ_p$-modules, 
$0<t_p<n$.
In \cite[Lemma 4.1]{s2}, we show that after appropriate extension of scalars, the $\bZ_p$-divisible part $D_p(A)$
of $\overline{G}_{A,p}=G_A\otimes_{\bZ}\bZ_p$ is generated by eigenvectors of $A$ corresponding to eigenvalues divisible 
by a prime ideal $\fp$ above $p$. Let $F$ be a finite Galois extension of $\bQ$ containing all the eigenvalues of $A$,  
{\it e.g.}, $F$ is the splitting field of the characteristic polynomial $h_A$ of $A$; 
$$
\bQ\subset K=\bQ(\la)\subseteq F\subset\overline{\bQ}.
$$ 
Let $\fp$ be a prime ideal of the ring of integers $\cO_F$ of $F$ above $p$. 
Denote by $\Sg$ the set of all distinct eigenvalues of $A$ and let $P$ denote the set of all $\la\in\Sg$ divisible by $\fp$. Since $h_A$ is irreducible, the cardinalities are $|\Sg|=n$ and  $|P|=t_p$. Denote
$$
U_P=\oplus_{\la\in P}\Span_F({\bf u}(\la)),
$$ 
where ${\bf u}(\la)\in F^n$ is an eigenvector of $A$ corresponding to $\la$. Thus, $U_P$ is the span of all eigenvectors of $A$ corresponding to eigenvalues divisible by $\fp$ and $U_P=\cX_{A,\fp}$ in the notation of Section \ref{ss:eigen}. An endomorphism $T$ of $G_A$ induces a $\bZ_p$-module endomorphism of  $\overline{G}_{A,p}$ and therefore, $T(D_p(A))\subseteq D_p(A)$. This implies 
$T(U_P)\subseteq U_P$. Thus, there exists  a non-empty subset $S\subseteq\Sg$ with the smallest cardinality satisfying $T(U_S)\subseteq U_S$. Denote $G=\Gal(\overline{\bQ}/\bQ)$. It is not hard to see that for any 
$R,V\subseteq \Sg$ and $\sg\in G$ we have
\begin{alignat}{2}
U_R\cap U_V&= U_{R\,\cap\, V}, & \quad \sg(U_R)&= U_{\sg(R)}  \label{eq:syst11}.
\end{alignat}
Assume $T(U_N)\subseteq U_N$ for some non-empty $N\subseteq \Sg$ and let $\sg\in G$. 
Since
$T$ is defined over $\bQ$,  
using properties
\eqref{eq:syst11}, we have $T(U_{\sg(N)})\subseteq U_{\sg(N)}$.
Hence,
$T(U_{S\,\cap\,\sg(N)})\subseteq U_{S\,\cap\,\sg(N)}$.
Since $S$ is the smallest with this property, either $S\,\cap\,\sg(N)=S$ or
$S\,\cap\,\sg(N)=\emptyset$. Equivalently, $\sg(S)\,\cap\,N=\sg(S)$ or
$\sg(S)\,\cap\,N=\emptyset$. In particular, taking $N=\tau(S)$
for an arbitrary $\tau\in G$, either $\sg(S)=\tau(S)$ or
$\sg(S)\,\cap\,\tau(S)=\emptyset$. Moreover, since 
$h_A$  is irreducible, $G$ acts transitively on $\Sg$. 
This implies that $N$ is a disjoint union of orbits $\sg(S)$ of $S$, $\sg\in G$, and, furthermore, there exists a subset $H\subseteq G$ depending on $N$ such that 
\bbe\label{eq:yay}
N=\bigsqcup_{\sg\in H} \sg(S),\quad |N|=|H|\cdot |S|.
\ee
Clearly, $T(U_N)\subseteq U_N$ holds for $N=\Sg$ and also for $N=P$. Thus, by \eqref{eq:yay}, $|S|$ divides both 
$n$ and $t_p$. By assumption, $(n,t_p)=1$ and hence $|S|=1$. Therefore, there exists an eigenvector ${\bf u}$ (corresponding to an eigenvalue $\la$) of $A$ such that $T({\bf u})=x{\bf u}$ for some $x\in\overline{\bQ}$. For a fixed eigenvalue $\la$, we can choose ${\bf u}\in K^n$, $K=\bQ(\la)$, and hence $x\in K$.  
Multiplying ${\bf u}$ by an appropriate integer, without loss of generality, we can assume ${\bf u}\in(\cO_K)^n$. 
Since $h_A$ is irreducible,  
$G$ acts transitively on the set of all eigenvalues of $A$, {\it i.e.}, there exist
$\sg_1,\ldots,\sg_n\in G$, $\sg_1=\id$, such that $A=M\La M^{-1}$, where
\bbe\label{eq:lm}
\La=\diag\left(\begin{matrix} \sg_1(\la)& \ldots & \sg_n(\la)\end{matrix}\right),\quad M=\left(\begin{matrix} \sg_1({\bf u}) & \ldots & \sg_n({\bf u}\end{matrix})\right)
\ee
with each $\sg_i({\bf u})$ written as a column, $i\in\{1,\ldots,n\}$. Then
\bbe\label{eq:tx}
T=MXM^{-1},\quad X=\diag\left(\begin{matrix} \sg_1(x)&\ldots &\sg_n(x)\end{matrix}\right).
\ee 
Thus, if $\la$ is fixed, then 
$T$ is completely determined by $x\in K$. A different choice of $\la$, {\it e.g.}, $\sg(\la)$ for some 
$\sg\in G$, will result in $\sg(x)$. We fix an eigenvalue $\la$ of $A$ and let $M$ be given by \eqref{eq:lm}. 
Define $\imath:\End(G_A)\rar K$ via $\imath(T)=x$, $T\in\End(G_A)$. By above, $\imath$ is an injective ring homomorphism. Note that $M\in\M_n(\cO_F)$. 
By \cite[p. 4, Theorem 2]{nt}, there exists a finite extension $L\subset\overline{\bQ}$ of $F$ and 
$P\in\GL_n(\cO_L)$ such that $PM$ is upper-triangular. From the definition of $G_A$, there exists $i\in\bN\cup\{0\}$ such that $A^iT\in\M_n(\bZ)$. In particular, $\la^ix\in\cO_L\cap K=\cO_K$ and hence $x\in\cO_K[\la^{-1}]=\cO_{K,{\la}}$.
\end{proof}
It is well-known that $\cU_{K,\cS}$ is a finitely generated abelian group. Therefore, by Proposition \ref{pr:end},
$\Aut(G_A)$ is also finitely-generated. 
\begin{cor}\label{cor:abel}
Under the assumptions of Proposition $\ref{pr:end}$, 
$\End(G_A)$ is a commutative ring, and $\Aut(G_A)$ is a finitely generated abelian group. 
\end{cor}

Clearly, $A^i\in\End(G_A)$ for any $i\in\bZ$. Since
$\imath(A^i)=\la^i$, this implies that $\bZ[\la^{\pm 1}]=\bZ[\la,\frac{1}{\la}]\subseteq \imath(\End(G_A))$ and $\imath(\End(G_A))$ (equivalently, $\End(G_A)$) is a $\bZ[\la^{\pm 1}]$-module (equivalently, a $\bZ[t^{\pm 1}]$-module via $\la\mapsto t$, $t$ is a variable). Thus,
$$
\bZ[\la^{\pm 1}]\subseteq \imath(\End(G_A))\subseteq \cO_{K,{\la}}=\cO_{K}[\la^{-1}]. 
$$
Moreover, under the assumptions of Proposition $\ref{pr:end}$, 
$\End(G_A)$ is a finitely-generated $\bZ[\la^{\pm 1}]$-module. Indeed, we have that 
$$
\bZ[\la]\subseteq Y\subseteq\cO_K,\quad Y=\imath(\End(G_A))\cap \cO_K.
$$
 It is well-known that both $\cO_K$, $\bZ[\la]$ are finitely-generated $\bZ$-modules of rank $n$ and therefore so is 
$Y$. Let $s=[Y:\bZ[\la]]$, and let $\ga_1,\ldots,\ga_s\in Y$ be representatives of $Y/\bZ[\la]$.
Let $T\in\End(G_A)$, $\imath(T)=x\in\cO_{K,{\la}}$, so that $y=\la^ix\in\cO_K$ for some $i\in\bN\cup\{0\}$.
Hence, $y\in Y$ and 
$y=\ga+a$ for some $\ga\in\{\ga_1,\ldots,\ga_s\}$ and $a\in\bZ[\la]$. Then, 
$$
x=\la^{-i}y=\ga\la^{-i}+a\la^{-i},\quad a\la^{-i}\in\bZ[\la^{\pm 1}], 
$$ 
{\it i.e.}, 
$1,\ga_1,\ldots,\ga_s$ generate $\imath(\End(G_A))$ over $\bZ[\la^{\pm 1}]$. This proves the following 

\begin{cor}\label{cor:integ}
Under the assumptions of Proposition $\ref{pr:end}$, $\imath(\End(G_A))$ $($equivalently, $\End(G_A))$ 
is a finitely-generated $\bZ[\la^{\pm 1}]$-module. 
If $\cO_K=\bZ[\la]$, then 
$$
\imath(\End(G_A))=\bZ[\la^{\pm 1}]=\cO_{K,{\la}}.
$$
\end{cor}

\begin{cor}\label{cor:mine} 
Under the assumptions of Proposition $\ref{pr:end}$, 
$$\End(G_A)\otimes_{\bZ}\bQ\cong\bQ(\la).$$
\end{cor}

\subsection{$2$-dimensional case}\label{ss:22irr}
The approach in the proof of Proposition \ref{pr:end} can be made more precise. We demonstrate it in the case $n=2$. 
Assume $A\in\M_2(\bZ)$ is non-singular with an irreducible characteristic polynomial $h_A\in\bZ[x]$ and $\la\in\overline{\bQ}$ is a root of $h_A$. Also, let  
$\cP'\ne\emptyset$, equivalently, there exists a prime $p\in\bN$ that divides $\det A$ and does not divide $\tr A$. 
In the notation of the proof of Proposition \ref{pr:end}, we have that
\begin{eqnarray}\label{eq:matr}
&& A=M\left(\begin{matrix}
\la & 0 \\
0 & \sg(\la)
\end{matrix}
\right)M^{-1},\quad 
M=\left(\begin{matrix}
{\bf u} & \sg({\bf u})
\end{matrix}
\right),  \nonumber \\
&& X=\left(\begin{matrix}
x & 0 \\
0 & \sg(x)
\end{matrix}
\right), \quad T(x)=MXM^{-1},
\end{eqnarray}
where ${\bf u}\in(\cO_K)^2$ is an eigenvector of $A$ corresponding to $\la$ written as a column,  $K=\bQ(\la)$ is a quadratic extension of $\bQ$, 
$x\in K$, and $\sg\in\Gal(K/\bQ)$ is the only non-trivial element. Moreover, there exist a finite extension  $L$ of $K$ and $P\in\GL_2(\cO_L)$ such that
\bbe\label{eq:t2}
PM=\left(\begin{matrix}
1 & u \\
0 & v
\end{matrix}
\right),\quad
PTP^{-1}=
\left(\begin{matrix}
x & w \\
0 & \sg(x)
\end{matrix}
\right), \quad w(x)=\frac{u(\sg(x)-x)}{v},
\ee
where 
$u,v\in \cO_L$, $\sg(\la)-\la=vv'$ for some $v'\in\cO_L$, and the ideal generated by $u$ and $v$ in $\cO_L$ is $\cO_L$, denoted by $(u,v)=\cO_L$. This follows from the fact that for any number field $K$ there exists a finite extension $L$ of $K$ such that every ideal of $\cO_K$ becomes principal in $\cO_L$ (see {\it e.g.}, \cite[p. 4, Theorem 2]{nt} and \cite[Corollary A.2]{s1}). In particular, if $\cO_K$ is a principal ideal domain, then one can take $L=K$. Denote by $\cS'$ the set of  all prime ideals of $\cO_L$ lying above all primes in $\cP$. 

\sbr


\begin{prop}\label{pr:discr}
Assume $A\in\M_2(\bZ)$ is non-singular with an irreducible characteristic polynomial $h_A\in\bZ[x]$ and $\cP'\ne\emptyset$. Then
\bbe\label{eq:imath2}
\imath(\End(G_A))=\{x\in\cO_{K,{\la}}\,\vert\, T(x)=MXM^{-1}\in\M_2(\caR)\}.
\ee
Let $[\cO_K:\bZ[\la]]=l_1l_2$, where $l_1,l_2\in\bN$, $\rad(l_1)$ divides $\det A$, $(l_2,\det A)=1$. 
Let $K=\bQ(\sqrt{d})$, where $d\in\bZ$ is square-free, and let $\{1,\om\}$ be the integral basis of $\cO_K$ with $\om=(1+\sqrt{d})/2$ if 
$d\equiv 1\,(\text{mod }4)$ and $\om=\sqrt{d}$ otherwise. Then $\imath(\End(G_A))$ is generated over $\bZ[\la^{\pm 1}]$ by $\{1,\al\om\}$, where $\al\in\bN$ divides $l_2$. In particular, $\al$ is the smallest natural number
such that 
\bbe\label{eq:val}
\frac{\al(\sg(\om)-\om)}{v}\in\cO_{L,\cS'}.
\ee
\end{prop}
\begin{proof}
Let $x\in\cO_{K,\la}$ and let $T=MXM^{-1}$, where $M$ and $X$ are given by \eqref{eq:matr}. 
It follows from the definition of $\cO_{K,\la}$ that $x=y\la^{-i}$ for some $y\in\cO_K$ and $i\in\bN\cup\{0\}$. 
Moreover, 
$T=T(x)\in\End(G_A)$ if and only if $T\in\M_2(\caR)$. Indeed, the necessary part follows from Section \ref{ss:loc}. To prove the sufficient part, By Lemma \ref{l:simple}, it is enough to show that $T\in\End(\overline{G}_{A,p})$ for any $p\in\cP'$. For any $p\in\cP'$ there exists a prime ideal $\fp$ of $\cO_L$ above $p$ such that $\fp$ divides $\la$ and $\fp$ does not divide $\sg(\la)$. Let $L_{\fp}$ denote the completion of $L$ with respect to $\fp$ with its ring of integers $\cO_{\fp}$. We have that $T(\overline{G}_{A,p})\subseteq \overline{G}_{A,p}$ if and only if $T(\overline{G}_{A,p}\otimes_{\bZ_p}\cO_{\fp})\subseteq \overline{G}_{A,p}\otimes_{\bZ_p}\cO_{\fp}$
if and only if $PTP^{-1}V\subseteq V$ with $V=P(\overline{G}_{A,p}\otimes_{\bZ_p}\cO_{\fp})$. Note that 
$$
V=\left\{(PAP^{-1})^{k}{\bf x}\, \vert \, {\bf x}\in\cO_{\fp}^2,\,k\in\bZ\right\}
=\left\{\left(\begin{matrix} \al & \be \end{matrix}\right)^t \, \vert \, \al\in L_{\fp},\,\be\in\cO_{\fp}\right\}=L_{\fp}\oplus\cO_{\fp},
$$
since $P\in\GL_2(\cO_L)$. Now it is clear that $PTP^{-1}V\subseteq V$, since $PTP^{-1}$ has the form \eqref{eq:t2} and 
$\sg(x)=\sg(y)\sg(\la)^{-i}\in\cO_{\fp}$. Indeed, $\sg(y)\in\cO_K$, $\cO_K\hrar\cO_L\hrar \cO_{\fp}$ and $\sg(\la)$ is a unit in $\cO_{\fp}$, since $\fp$ does not divide $\sg(\la)$. This shows that $T$ defined by \eqref{eq:matr} with $x\in\cO_{K,\la}$ belongs to $\End(G_A)$ if and only if $T\in\M_2(\caR)$, {\it i.e.}, \eqref{eq:imath2} holds. 

\sbr

We have that $x=y\la^{-i}$, $y\in\cO_K$, and hence $y=a+b\om$ for $a,b\in\bZ$. 
Since $\bZ[\la^{\pm 1}]\subseteq \imath(\End(G_A))$,
$x\in\imath(\End(G_A))$ if and only if $b\om\in\imath(\End(G_A))$ if and only if 
$T(b\om)\in\M_2(\caR)$ by  \eqref{eq:imath2}  if and only if $w(b\om)\in\cO_{L,\cS'}$ by \eqref{eq:t2} if and only if
\bbe\label{eq:xx}
\frac{b(\sg(\om)-\om)}{v}\in\cO_{L,\cS'},
\ee 
since $(u,v)=\cO_L$ by assumption. 
It is well-known and one can also easily check that $\sg(\la)-\la=\pm[\cO_K:\bZ[\la]](\sg(\om)-\om)$.
Since $v$ divides $\sg(\la)-\la$ by \eqref{eq:t2} and $l_1$ is a unit in $\cO_{L,\cS'}$, \eqref{eq:xx} holds for $b=l_2$. Also, the set 
$$
I=\{b\in\bZ\,\vert\,b\om\in\imath(\End(G_A)) \}
$$ 
 is an ideal of $\bZ$, 
$l_2\in I$, and therefore 
$I=(\al)$ is generated by the smallest $\al\in\bN$ such that $\al\in I$. In particular, $\al$ divides $l_2$.
\end{proof}

\begin{cor}\label{cor:discr}
Assume $A\in\M_2(\bZ)$ is non-singular with an irreducible characteristic polynomial $h_A\in\bZ[x]$ and $\cP'\ne\emptyset$. If $\rad[\cO_K:\bZ[\la]]$ divides $\det A$, then $$\imath(\End(G_A))=\cO_{K,{\la}}.$$
\end{cor}

\begin{proof}
If $\rad[\cO_K:\bZ[\la]]$ divides $\det A$, then in the notation of Proposition \ref{pr:discr}, $l_2=1$ and hence $\al=1$. Then, by Proposition \ref{pr:discr}, 
$$\imath(\End(G_A))=\left\{c+d\om\,\vert\, c,d\in \bZ [\la^{\pm 1}] \right\}=\cO_{K,{\la}}.$$
\end{proof}
\begin{rem}\label{rem:index}
Note that if $D=m^2\cdot d$ is the discriminant of $h_A$, where $m\in\bN$ and 
$d\in\bZ$ is square-free, then 
$$
[\cO_K:\bZ[\la]]=\begin{cases} m, & d\equiv 1\,(\text{mod }4) \\
\frac{m}{2}, & \text{otherwise}.
\end{cases}
$$ 
%
%
%
%
\end{rem}

\begin{example}
In this example we show that the endomorphism ring of $G_A$ does not determine $G_A$ up to an isomorphism, {\it i.e.}, there exist $A,B\in\M_n(\bZ)$ such that $\End(G_A)\cong\End(G_B)$ as rings, but
$G_A\not\cong G_B$ as groups.
Consider a quadratic number field $K=\bQ(\la)$ defined by a root $\la\in\overline{\bQ}$ of an irreducible polynomial $h=x^2-x+13$ \cite[Number field 2.0.51.1]{db}. Since the class group of $K$ has order $2$, there are two $\operatorname{GL}_2(\bZ)$-conjugacy classes of matrices $[A]$, $[B]$, $A,B\in\M_2(\bZ)$, with $A$ corresponding to the trivial ideal $\cO_K$ and $B$ corresponding to the non-trivial ideal $I=\bZ[3,\la+1]$, a generator of the class group. For example,
$$
A=\left(\begin{matrix}
0 & 1 \\
-13 & 1
\end{matrix}
\right),\quad
B=\left(\begin{matrix}
-1 & 3 \\
-5 & 2
\end{matrix}
\right).
$$
Both $A$ and $B$ share the same characteristic polynomial $h$ with eigenvalues $\la_{1}$, $\la_2$, so that $\tr A=\tr B=1$, $\det A=\det B=13$, the discriminant $D=-3\cdot 17$ is square-free, $\la_{1,2}=\frac{1\pm\sqrt{D}}{2}$, 
$\caR=\caR(A)=\caR(B)=\{r13^s\,\vert \, r,s\in\bZ\}$. Moreover, $A$, $B$ are conjugated by a matrix from $\GL_2(\bQ)$, but there is no matrix $S\in\operatorname{GL}_2(\bZ)$ such that $A=SBS^{-1}$. By Corollary \ref{cor:integ} and Remark \ref{rem:index}, 
$\imath(\End(G_A))=\imath( \End(G_B))=\cO_{K,\la}$, where $\imath$ is defined by the choice of $\la$ ($\la_1$ or $\la_2$). Thus, 
$\End(G_A)\cong \End(G_B)$ as rings. However, $G_A\not\cong G_B$ as groups. 
Indeed, assume $G_A\cong G_B$. By \cite[p. 207, Corollary 6.3]{s1}, $G_A\cong G_B$ if and only if there exists $T\in\GL_2(\caR)$ such that $A=TBT^{-1}$, where $T=MXN^{-1}$, $X=\diag\left(\begin{matrix} x&\sg(x)\end{matrix}\right)$
for some $x\in K$, $A=M\La M^{-1}$, $B=N\La N^{-1}$, $\La=\diag\left(\begin{matrix} \la&\sg(\la)\end{matrix}\right)$,
$$
M=\left(\begin{matrix}
1 & 1 \\
\la & \sg(\la)
\end{matrix}
\right),\quad
N=\left(\begin{matrix}
3 & 3 \\
\la+1 & \sg(\la)+1
\end{matrix}
\right).
$$
In particular, $\det T=N_{K/\bQ}(x)\cdot\det M\cdot(\det N)^{-1}=N_{K/\bQ}(x)\cdot 3^{-1}$, where $N_{K/\bQ}(x)$ denotes the norm of 
$x$. Since $T\in\GL_2(\caR)$, $\det T\in\caR^{\times}$, which implies 
\bbe\label{eq:norm}
N_{K/\bQ}(x)=\pm 13^k\cdot 3, \quad k\in\bZ.
\ee
It is known that $(13)=I_1\cdot I_2$, where $I_1,I_2$ are principal ideals of $\cO_K$ and $(3)=I^2$ \cite{sage}.  
Thus, it follows from \eqref{eq:norm} that $(x)=I_1^s\cdot I_2^t\cdot I$ for some $s,t\in\bZ$, which implies that $I$ is principal. This is a contradiction, since $I$ has order $2$ in the class group of $K$ and hence it is not principal. Thus, $G_A\not\cong G_B$.
\end{example}

\begin{example}
In this example, we show that the condition ``$\rad[\cO_K:\bZ[\la]]$ divides $\det A$"  in Corollary \ref{cor:discr}
 is not necessary for 
$\imath(\End(G_A))=\cO_{K,{\la}}$. Here, in the notation of Proposition \ref{pr:discr}, $l_2\ne 1$ and $\al=1$.  Let
$$
A=\left(\begin{matrix}
-1 & 3 \\
3 & 2
\end{matrix}
\right),\quad h_A(x)=x^2-x-11,
$$
$D=3^2\cdot 5$ is the discriminant of $h_A$. Hence, by Remark \ref{rem:index}, $m=3$, $d=5$, 
$\om=\frac{1+\sqrt{5}}{2}$, $\la=\frac{1+3\sqrt{5}}{2}$, $[\cO_K:\bZ[\la]]=3$, and $l_2=3$. Note that $\cO_K$ is a principal ideal domain, hence, in the notation of Section \ref{ss:22irr}, $L=K$, 
$$
M=\left(\begin{matrix}
1 & 1 \\
\om & \sg(\om)
\end{matrix}
\right),\quad 
PM=\left(\begin{matrix}
1 & 1 \\
0 & \sg(\om)-\om
\end{matrix}
\right),\quad P\in\GL_2(\cO_K),
$$
and $v=\sg(\om)-\om$. Thus, \eqref{eq:val} holds for $\al=1$ and $\imath(\End(G_A))=\cO_{K,{\la}}=\cO_K[\la^{-1}]$.
\end{example}

\begin{example} In this example, we demonstrate how \eqref{eq:imath2} can be used to determine 
$\End(G_A)$ for a rational canonical form $A\in\M_2(\bZ)$ of a monic irreducible quadratic polynomial 
$h_A=x^2+\be x+\ga\in\bZ[x]$. By Lemma \ref{l:easy}, $\End(G_A)=\M_2(\bZ)$ if $\ga=\pm 1$,
$\End(G_A)=\M_2(\caR)$ if $\ga\ne\pm 1$ and $\cP'=\emptyset$. Assume, $\cP'\ne\emptyset$, equivalently,
$\rad(\ga)$ does not divide $\rad(\be)$. Then, Proposition \ref{pr:discr} can be applied. 
Let
$$
A=\left(\begin{matrix}
0 & -\ga \\
1 & -\be
\end{matrix}
\right), \quad 
M=\left(\begin{matrix}
-\ga & -\ga \\
\la & \sg(\la)
\end{matrix}
\right),\quad A=M\left(\begin{matrix}
\la & 0 \\
0 & \sg(\la)
\end{matrix}
\right)M^{-1}.
$$
For $x=b\om$, $b\in\bZ$, one can check that $T(x)=MXM^{-1}\in\M_2(\caR)$ if and only if $l_2$ divides $b$.
Therefore, by Proposition \ref{pr:discr}, $\imath(\End(G_A))$ is generated by $\{1,l_2\om\}$ as a $\bZ[\la^{\pm 1}]$-module. 
More precisely, if $T_0=MX_0M^{-1}$ with $X_0=\diag\left(\begin{matrix} l_2\om& l_2\sg(\om)\end{matrix}\right)$,
then
$$
\End(G_A)=\left\{\sum_i b_iA^{m_i}+\sum_j c_jA^{n_j}T_0\,\,\Big\vert \,\, \forall \, b_i,m_i,c_j,n_j\in\bZ,\, i,j\in\bN  \right\},
$$
where each sum has finitely many non-zero terms. 
Since $1$ and $l_2\om$ are $\bZ[\la^{\pm 1}]$-dependent, $\End(G_A)$ is a finitely generated $\bZ[\la^{\pm 1}]$-module of rank $1$. 

%
\end{example}

\section{Character groups, solenoids, and $\cS$-integer dynamical systems} 
\subsection{Character groups} 
In this section, we describe the Pontryagin dual $\widehat{G_{A}}$ of $G_A$. 
Here, $G_A$ is considered as a topological group endowed with the discrete topology and 
 $\widehat{G_{A}}$ is a topological group with the underlying space consisting of continuous group homomorphisms from $G_A$ to a circle $\bT^1$ endowed with the compact-open topology.

\sbr

Let $A\in\M_n(\bZ)$ be non-singular
and let $h_A\in\bZ[t]$ be the characteristic polynomial of $A$. 
Let $h_A=h_1h_2\cdots h_s$,
where $h_1,\ldots,h_s\in\bZ[t]$ are irreducible of degrees $n_1,\ldots,n_s$, respectively.

\begin{lem}\cite[Lemma 8.1]{s2}\label{l:dense}
$G_A$ is dense in $\bR^n$ endowed with the standard topology if and only if 
$h_i(0)\ne\pm 1$ for all $i\in\{1,2,\ldots,s\}$.
\end{lem}

In the case when $G_A$ is {\it not} dense in $\bR^n$, by Lemma \ref{l:dense}, 
there exist $f_1,f_2\in\bZ[t]$ such that $h_A=f_1f_2$ and $f_2(0)=\pm 1$.
Let $g_2\in\bZ[t]$ be of maximal degree such that there exists 
$g_1\in\bZ[t]$ with $h_A=g_1g_2$, $g_2(0)=\pm 1$. In other words, if $h_A=h_1h_2\cdots h_s$, where $h_1,\ldots,h_s\in\bZ[t]$ are irreducible, $h_i(0)\ne\pm 1$ for all $i\in\{1,\ldots,t\}$ and $h_j(0)=\pm 1$
for all $j\in\{t+1,\ldots,s\}$, $1\leq t<s$, then $g_1=h_{1}\cdots h_t$ and $g_2=h_{t+1}\cdots h_s$.
Then there exists $S\in\GL_n(\bZ)$ such that 
\bbe\label{eq:A1}
SAS^{-1}=\left(\begin{matrix}
A_1 & * \\
0 & A_2
\end{matrix}\right),
\ee
where $A_1\in\M_{k}(\bZ)$ has characteristic polynomial $g_1$ and $A_2\in\M_{n-k}(\bZ)$ has 
characteristic polynomial $g_2$ 
\cite[p. 50, Thm. III.12]{n}. Thus, $G_{A_1}$ is dense in $\bR^k$ endowed with the standard  topology, $\det A_2=\pm 1$, and 
$G_{A_2}=\bZ^{n-k}$. Thus, the natural exact sequence 
$$0\rar G_{A_1}\rar G_{SAS^{-1}}\rar G_{A_2}\rar 0$$ splits, so that 
$$
S(G_{A})=G_{SAS^{-1}}\cong G_{A_1}\oplus\bZ^{n-k},\quad G_{A}\cong G_{A_1}\oplus\bZ^{n-k}.
$$
Therefore, 
\bbe\label{eq:hats}
\widehat{G_{A}}\cong \widehat{G_{A_1}}\oplus\widehat{\bZ}^{n-k}\cong \widehat{G_{A_1}}\oplus {\bT}^{n-k}.
\ee
Therefore, to study the character group $\widehat{G_{A}}$, it is enough to consider the case when $G_A$ is dense in $\bR^n$  endowed with the standard  topology. For ${\bf y}\in\bR^n$, 
denote by $\La({\bf y})\in\widehat{\bR^n}$ the character of $\bR^n$ given by 
$$
\La({\bf y})({\bf x})=e^{2\pi i{\bf y}\cdot{\bf x}},\quad {\bf x}\in\bR^n,
$$
where ${\bf y}\cdot{\bf x}$ is the standard dot product of vectors in $\bR^n$. We consider $\La({\bf y})$ as a character of 
$G_A$ via the restriction. Note that if $G_A$ is endowed with the topology $\tau$ induced from the standard topology on 
$\bR^n$ and $(G_A,\tau)$ is dense in $\bR^n$, then $\bR^n\cong\widehat{\bR^n}\cong\widehat{(G_A,\tau)}$ with the isomorphism given by ${\bf y}\mapsto \La({\bf y})$, ${\bf y}\in\bR^n$. However, if $G_A$ is endowed with the discrete topology, then the structure of $\widehat{G_{A}}$ is more complicated and is given by a quotient of an ad\`ele ring of $K$ (see Theorem \ref{th:onto} below).

\sbr

In the next lemma, we give a description of $\widehat{G_{A}}$ for an arbitrary non-singular $A\in\M_n(\bZ)$. 
For ${\bf y}\in\bQ_p^n$, we will denote by $\{ {\bf y}\}_p$ the ``fractional" part of ${\bf y}$, {\it i.e.}, ${\bf y}=\{ {\bf y}\}_p+{\bf y}_1$, where ${\bf y}_1\in\bZ_p^n$.

\begin{lem}\label{l:grchargen}
Let $A\in\M_n(\bZ)$ be non-singular  and let $\widehat{G_{A}}$ denote the Pontryagin dual of $G_A$, where $G_A$ is endowed with the discrete topology. 
Consider the following map:
$$
\psi:\bR^n\times\widehat{{G}_{A}/\bZ^n}\rar \widehat{G_{A}},\quad
\psi({\bf\theta},\chi')=\La(-{\bf\theta})\chi',\quad {\bf\theta}\in\bR^n,
$$
where $\chi'\in \widehat{{G}_{A}/\bZ^n}$ is considered as a character of $G_A$ trivial on $\bZ^n$. 
If $G_A$ is dense in $\bR^n$ with respect to the standard topology on $\bR^n$, then the map $\bZ^n\rar \widehat{{G}_{A}/\bZ^n}$ 
given by ${\bf m}\mapsto \La({\bf m})$ is an embedding and $\psi$ induces a group isomorphism 
\bbe\label{eq:mine1} 
\left(\bR^n\times\widehat{{G}_{A}/\bZ^n}\right)\Big/ \bZ^n\cong \widehat{G_{A}},
\ee
where $\bZ^n$ is embedded into the product as ${\bf m}\mapsto({\bf m},\La({\bf m}))$. 
Moreover,
$$
\widehat{{G}_{A}/\bZ^n}\cong\prod_{p\,\vert \det A} \bZ_{p}^{t_p},
$$
and for every $\chi'\in\widehat{G_A/\bZ^n}$ there exist ${\bf v}_p\in\bZ_p^n$, $p\,\vert \det A$, such that
$$
\chi'({\bf x})=\prod_{p\,\vert \det A} e^{2\pi i\{{\bf x}\cdot{\bf v}_p\}_p}.
$$
\end{lem}
\begin{proof}
As in \cite{con}, for $\chi\in\widehat{G_{A}}$,  let 
$\chi({\bf e}_1)=e^{-2\pi i\theta_1},\ldots, \chi({\bf e}_n)=e^{-2\pi i\theta_n}$, where 
$\{{\bf e}_1,\ldots, {\bf e}_n\}$ is the standard basis of $\bR^n$.
Define $\chi'\in\widehat{G_{A}}$ via 
\bbe\label{eq:infloc}
{\bf\theta}=\left(\begin{matrix}\theta_1 &\ldots & \theta_n \end{matrix}\right)\in\bR^n,\quad 
\chi'=\La({\bf \theta})\chi.
\ee
Then $\chi'$ is trivial on $\bZ^n$, {\it i.e.}, $\chi'\in\widehat{G_A/\bZ^n}$, and hence 
$\psi$ is onto. We now find the kernel of $\psi$. Let $\chi=\La(-{\bf\theta})\chi'$ be trivial on $G_A$. Then,
$\La(-{\bf\theta})$ is trivial on $\bZ^n$, since $\chi'$ is trivial on $\bZ^n$ by assumption. Hence, ${\bf\theta}\in\bZ^n$ and $\chi'=\La({\bf\theta})$. Finally, if $\La({\bf m})$ is trivial on $G_A$ and 
$G_A$ is dense in $\bR^n$ in the standard topology, then $\La({\bf m})$ is trivial on $\bR^n$, since
$\La({\bf m})$ is a continuous character of $\bR^n$ in the standard topology. This implies that 
${\bf m}={\bf 0}$ and hence ${\bf m}\mapsto \La({\bf m})$ defines an embedding of $\bZ^n$ into 
$\widehat{{G}_{A}/\bZ^n}$. 

\sbr

We know that there is an isomorphism
$$
\psi_A:\prod_{p\,\vert \det A} \overline{G}_{A,p}/\bZ_{p}^n\rarab{\sim} G_{A}/\bZ^n
$$
induced by $\psi_A({\bf v}_p)=\{{\bf v}_p\}_p$, ${\bf v}_p\in\overline{G}_{A,p}$ \cite[Lemma 3.2]{s2}.
Let $\chi'\in\widehat{G_{A}/\bZ^n}$.  Then 
\bbe\label{eq:chargr1}
\widehat{G_{A}/\bZ^n}\cong\prod_{p\,\vert \det A}\widehat{\overline{G}_{A,p}/\bZ_{p}^n},\quad
\chi'=\prod_{p\,\vert\det A}\chi_p',\quad\forall \chi_p'\in \widehat{\overline{G}_{A,p}/\bZ_{p}^n}.
\ee
We now fix a prime $p$ dividing $\det A$ and describe a 
character $\chi_p'$.
By \cite[Lemma 2.10]{s1}, $\overline{G}_{A,p}/\bZ_{p}^n\cong \left(\bQ_p/\bZ_p\right)^{t_p}$ and it is well-known that
$\widehat{\bQ_p/\bZ_p}\cong\bZ_p$. Thus,
$$
\widehat{\overline{G}_{A,p}/\bZ_{p}^n}\cong\widehat{\left(\bQ_p/\bZ_p\right)^{t_p}}\cong \bZ_p^{t_p}.
$$
Tracing the maps, one can show that each character $\chi_p'$ of $\overline{G}_{A,p}$ trivial on $\bZ_p^n$ is determined by ${\bf v}_p\in\bZ_p^n$ via
$$
\chi_p'({\bf x})=e^{2\pi i\{{\bf x}\cdot {\bf v}_p\}_p},\quad {\bf x}\in \overline{G}_{A,p}.
$$
\end{proof}

%

\subsection{$G_A$ as a subgroup of a number field}\label{ss:nf} So far, we have considered $G_A$ as a subset of $\bQ^n$. We now show that when the characteristic polynomial of $A$ is irreducible, one can consider $G_A$ as a subset of a number field $\bQ(\la)$, where $\la\in\overline{\bQ}$ is an eigenvalue of $A$.

\sbr

For the rest of the section we will fix an eigenvalue $\la\in\overline{\bQ}$ of $A$ and a corresponding eigenvector ${\bf u}\in(\overline{\bQ})^n$. Let $K=\bQ(\la)$ and, without loss of generality, we can assume that 
${\bf u}=\left(\begin{matrix} u_1 & \ldots & u_n \end{matrix}\right)\in(\cO_K)^n$, 
where $\cO_K$ denotes the ring of integers of $K$.
Let $\cS_{\la}$ consist of (all) prime ideals of $\cO_K$ dividing $\la$. (Note that $\la\in\cO_K$.) 
Recall that $\cO_{K,{\la}}$ denotes the ring of $\cS_{\la}$-integers, {\it i.e.},
$$
\cO_{K,{\la}}=\{x\in K\,\vert\, \val_{\fp}(x)\geq 0\text{ for any prime ideal }\fp\text{ of }\cO_K\text{ not in }\cS_{\la}\}
=\cO_{K}\left[{\la}^{-1}\right].
$$
Assume $h_A\in\bZ[t]$ is irreducible. Then $\Gal(\overline{\bQ}/\bQ)$ acts transitively on all the eigenvalues 
$\la_1=\la,\ldots,\la_n$ of $A$, {\it i.e.}, $\la_i=\sg_i(\la)$
for embeddings $\sg_1=\id,\sg_2,\ldots,\sg_n\in\Gal(\overline{\bQ}/\bQ)$ of $K$ into $\overline{\bQ}$. 
Then ${\bf u}_i=\sg_i({\bf u})$ is an eigenvector of $A$ corresponding to $\la_i$, $i\in\{1,\ldots,n\}$. For ${\bf x}\in G_A$, 
since ${\bf u}_1={\bf u},\ldots,{\bf u}_n$ are linearly independent over $\overline{\bQ}$, ${\bf x}=\sum_{i=1}^n x_i{\bf u}_i$ for some $x_1,\ldots,x_n\in \overline{\bQ}$.
Since ${\bf x}\in\bQ^n$, $\sg_i({\bf x})={\bf x}$ and hence $x_i=\sg_i(x_1)$  for all $i$. Note that $x_1\in K$.
Indeed, since 
${\bf u}_1={\bf u}\in K^n$, for any $\sg\in\Gal(\overline{\bQ}/K)$, we have that 
$\sg({\bf u}_1)={\bf u}_1$, $\sg({\bf x})={\bf x}$, and hence $\sg(x_1)=x_1$ and $x_1\in K$. Thus, the projection
$\mu$ along ${\bf u}$ defines an injective homomorphism
\bbe\label{eq:mu}
\mu:G_A\hrar K,\quad \mu({\bf x})=x_1,\quad {\bf x}=\sum_{i=1}^n\sg_i(x_1{\bf u}),\quad\sg_1=\id.
\ee
To prove our main result in this section, Theorem \ref{th:onto} below, we will need the following lemma. 
\begin{lem}\label{l:mu1}
Assume $A\in\M_n(\bZ)$ is non-singular with an irreducible characteristic polynomial $h_A\in\bZ[x]$. 
Then, $\cO_{K,{\la}}\subseteq \mu(G_A)$, where $\mu$ is given by $\eqref{eq:mu}$.
\end{lem}
\begin{proof}
Let $x_1\in\cO_{K,{\la}}$ and let ${\bf x}=\sum_i\sg_i(x_1{\bf u}_1)$. We need to show that ${\bf x}\in G_A$. By construction, ${\bf x}\in\bQ^n$. Moreover, ${\bf x}\in\caR^n$, since for any 
$\sg\in\Gal(\overline{\bQ}/\bQ)$ and any $\fp\in\cS_{\la}$, 
$\sg(\fp)$ is a prime ideal of $\bQ(\sg(\la))$ dividing $\sg(\la)$. Thus, $\val_{\fq}\sg_i(x_1{\bf u}_1)\geq 0$ for any prime ideal $\fq$ of the splitting field $L$ of $h_A$ not dividing $\det A$.  
In other words, in the ``denominators" of $\sg_i(x_1{\bf u}_1)$'s there are only prime ideals dividing $\la_i$'s and hence in the denominators of ${\bf x}\in\bQ^n$ we only have primes 
$p\in\bN$ dividing $\det A$.  Thus, by \eqref{eq:fuks}, we only need to show that ${\bf x}\in\overline{G}_{A,p}$ for any $p\in\cP$ under the embedding induced by $\bQ\hrar{\bQ}_{p}$. 

\sbr

We fix an arbitrary $p\in\cP$. Even though $h_A$ is irreducible over $\bQ$, it might not be irreducible over 
$\bQ_p$. By the definition of $t_p$, we have $h_A(x)\equiv f(x)x^{t_p}\,(\text{mod }p)$, where $f\in\bF_p[x]$, $f(0)\ne 0$. 
Therefore, by Hensel's lemma,  $h_A=h_1h_2$, where $h_1,h_2\in\bZ_p[t]$, $h_1\equiv f\,(\text{mod }p)$, $h_2\equiv x^{t_p}\,(\text{mod }p)$, and $p$ does not divide $h_1(0)$. 
 Let $L$ be a finite Galois extension of $\bQ$ containing all eigenvalues of $A$, {\it e.g.}, $L$ is the splitting field of $h_A$. Let $\fq$ be a prime ideal of $L$ lying above $p$. Without loss of generality, we can assume that $\fq$ divides $\la_1=\la,\ldots,\la_{t_p}$ in $\cO_L$, so that $\la_1,\ldots,\la_{t_p}$ are all roots of $h_2$ in $\overline{\bQ}_p$. Then, for any prime ideal $\fp$ of $\cO_K$ dividing $\la$, 
 $\sg_i(\fp)$ is not divisible by $\fq$ for any $t_p<i\leq n$. Indeed, clearly $\fq$ does not divide any $\sg_i(\fp)$ if 
$\fp$ lies above a prime $p'\in\bN$ not equal to $p$. If $\fp$ lies above $p$ and 
$\sg_j(\fp)$ is divisible by $\fq$ for some $j>t_p$, then $\fq$ divides $\sg_j(\la)$ and we have a contradiction with $\deg h_2=t_p$. This implies that $x_j=\sg_j(x_1)\in\cO_{\fq}$ for any $j>t_p$, since $x_1\in\cO_{K,{\la}}$ by assumption. 
Let ${\bf x}={\bf y}_1+{\bf y}_2$, where ${\bf y}_1=\sum_{i=1}^{t_p}\sg_i(x_1{\bf u}_1)$,
${\bf y}_2=\sum_{i=t_p+1}^{n}\sg_i(x_1{\bf u}_1)$. Let $G_p=\Gal(\overline{\bQ}_p/\bQ_p)$. Since any 
$\sg\in G_p$ permutes roots of $h_1$ (respectively, $h_2$), we have that $\sg({\bf y}_i)={\bf y}_i$, $i=1,2$.
Hence, ${\bf y}_2\in(\cO_{\fq})^n\cap\bQ_p^n$, since ${\bf u}_1\in\cO_K\subseteq\cO_L$. Thus, 
${\bf y}_2\in\bZ_p^n\subseteq \overline{G}_{A,p}$. Then, ${\bf y}_1\in\Span_{L_{\fq}}({\bf u}_1,\ldots, {\bf u}_{t_p})\cap\bQ_p^n$, where $\Span_{L_{\fq}}({\bf u}_1,\ldots, {\bf u}_{t_p})\cap\bQ_p^n=D_p(A)\subseteq \overline{G}_{A,p}$ by \cite[Lemma 4.1]{s2}. This shows that ${\bf x}\in\overline{G}_{A,p}$ for any $p\in\cP$.
\end{proof}

We now describe the image of $\mu$ inside $K$. Let $M=\left(\begin{matrix} \sg_1({\bf u}) & \ldots & \sg_n({\bf u})\end{matrix}\right)\in\M_n(\overline{\bQ})$ be as in \eqref{eq:lm}, where $\sg_1,\ldots,\sg_n$ are all embeddings of $K$ into $\overline{\bQ}$ and $\sg_1=\id$.
\begin{lem}\label{l:image}
Assume $A\in\M_n(\bZ)$ is non-singular with an irreducible characteristic polynomial $h_A\in\bZ[x]$. 
Let ${\bf w}=\left(\begin{matrix} w_1 & \ldots & w_n \end{matrix}\right)^t\in(\overline{\bQ})^n$ be the $1$st column of $(M^t)^{-1}$, that is ${\bf w}$ is an eigenvector of $A^t$ corresponding to $\la$. 
Denote 
$$
Y_{A^t}({\bf w},\la)=\{m_1\la^{k_1}w_1+\cdots +m_n\la^{k_n}w_n\,\vert\,m_1,\ldots,m_n,k_1,\ldots,k_n\in\bZ \}.
$$
Then $Y_{A^t}({\bf w},\la)$ is a $\bZ[\la^{\pm 1}]$-submodule of $K$,  
$\mu(G_A)=Y_{A^t}({\bf w},\la)$, and $G_A\cong Y_{A^t}({\bf w},\la)$.
\end{lem}
\begin{proof} By definition, ${\bf w}\cdot{\bf u}=1$, where we assume ${\bf u}\in K^n$. Since 
$A^t{\bf w}=\la{\bf w}$ and $A^t$ has integer entries, there is an eigenvector ${\bf w}'\in K^n$ of $A^t$ corresponding to $\la$. Thus, ${\bf w}=\al{\bf w}'$ for some $\al\in\overline{\bQ}$. Therefore, 
$$
{\bf w}\cdot{\bf u}=\al{\bf w}'\cdot{\bf u}=1.
$$
Since ${\bf w}'\cdot{\bf u}\in K$, 
this implies that $\al\in K$, ${\bf w}\in K^n$, and $Y_{A^t}({\bf w},\la)\subseteq K$. Clearly, 
$Y_{A^t}({\bf w},\la)$ is a $\bZ[\la^{\pm 1}]$-submodule of $K$. Moreover, 
$$
{\bZ}[{\bf w}]=\{m_1w_1+\cdots +m_nw_n\,\vert\,m_1,\ldots,m_n\in\bZ \}
$$ 
is a $\bZ[\la]$-module, since $A^t$ has integer entries. Then, any $y\in Y_{A^t}({\bf w},\la)$ has the form
$y=u\la^k$ for some $u\in {\bZ}[{\bf w}]$ and $k\in\bZ$.

\sbr

Let ${\bf x}\in G_A$, 
$\mu({\bf x})=x_1$. Then, by the definition of $\mu$, 
${\bf x}=\sum_i\sg_i(x_1{\bf u})$. It can be easily verified that $(M^t)^{-1}=\left(\begin{matrix} \sg_1({\bf w}) & \ldots & \sg_n({\bf w})\end{matrix}\right)$.  By the definition of $G_A$ and \eqref{eq:lm}, ${\bf x}\in G_A$ if and only if there exists $k\in\bZ$ and ${\bf m}\in\bZ^n$ such that 
$$
{\bf x}=A^k{\bf m}=M\La^k M^{-1}{\bf m}=M\left(\begin{matrix} x_1 & \sg_2(x_1) & \ldots & \sg_n(x_1)\end{matrix}\right)^t 
$$
 if and only if $x_1\in Y_{A^t}({\bf w},\la)$. 
\end{proof}

Denote 
\bbe\label{eq:ya}
Y_{A}({\bf u},\la)=\{m_1\la^{k_1}u_1+\cdots +m_n\la^{k_n}u_n\,\vert\,m_1,\ldots,m_n,k_1,\ldots,k_n\in\bZ \},
\ee
an analogue of $Y_{A^t}({\bf w},\la)$ for $A$ and ${\bf u}$. Then $Y_{A}({\bf u},\la)$ is a $\bZ[\la^{\pm 1}]$-submodule of $\cO_{K,{\la}}$, since $u_1,\ldots,u_n\in\cO_K$ by assumption. Let $m=(\det M)^2\in\bZ$ be the discriminant of the lattice
${\bZ}[{\bf u}]$.
Then we have the diagram
\bbe\label{eq:diagram}
\xymatrix{
Y_{A}({\bf u},\la)\ar[d]^{\cong} & \subseteq\cO_{K,{\la}}\subseteq & Y_{A^t}({\bf w},\la)\ar[d]^{\cong} & \rarab{\cdot m} 
Y_{A^t}(m{\bf w},\la) \subseteq\cO_{K,{\la}} \\
G_{A^t} & & G_A &
},
\ee
where the ``down" isomorphisms are given by $\mu$ applied to $G_{A^t}$, $G_A$.

\subsection{Character group of $G_A$ via ad\`eles} We will use the notation introduced in \cite{t}. 
Let $\cS_{\infty}$ denote the set of all infinite places of $K$. Denote
\begin{align*} 
&\text{complex }\fp\in\cS_{\infty} &  &\La_{\fp}(\xi)=-2\Re(\xi) & &\xi\in K_{\fp}=\bC \\
&\text{real }\fp\in\cS_{\infty} &  &\La_{\fp}(\xi)=-\xi & &\xi\in K_{\fp}=\bR \\
&\text{finite }\fp &  &\La_{\fp}(\xi)=\{\tr_{K_{\fp}/\bQ_p}\xi\}_p & &\xi\in K_{\fp}.
\end{align*}
Recall that $\cS_{\la}$ denotes the set of all finite places (prime ideals) of $K$ dividing $\la$ in $\cO_K$. Note that $\cS_{\infty}\cup\cS_{\la}$ is a finite set. Denote
$$
\bA_{K,\la}=\prod_{\fp\in \cS_{\infty}\cup\cS_{\la}} K_{\fp}.
$$
It is an object of the same nature as the ad\`ele ring $\bA_K$ of $K$, consisting of all elements 
$(\ldots,\eta_{\fp},\ldots)$, where $\fp$ runs through all the places of $K$, $\eta_{\fp}\in K_{\fp}$ for any $\fp$ and 
$\eta_{\fp}\in \cO_{\fp}$ for all but finitely many $\fp$. It is known that the character group of $K$ endowed with the discrete topology is isomorphic to the quotient of $\bA_K$ by $K$, where $K$ is embedded into $\bA_K$ diagonally via $\xi\mapsto (\ldots,\xi_{\fp},\ldots)$, each $\xi_{\fp}=\xi$ \cite{t}. 
From Section \ref{ss:nf}, if the characteristic polynomial of $A$ is irreducible, then $G_A$ can be 
considered as a subset of $K$ via $\mu$ in \eqref{eq:mu}. 
Each $\eta=(\ldots,\eta_{\fp},\ldots)\in \bA_{K,\la}$ defines a character $\chi$ of $K$ via
\bbe\label{eq:chi11}
\chi(\eta)(\xi)=\prod_{\fp\in \cS_{\infty}\cup\cS_{\la}} e^{2\pi i\La_{\fp}(\eta_{\fp}\xi)},\quad\xi\in K,
\ee
hence $\eta$ defines a character of $G_A$ via restriction and isomorphism $\mu$ in \eqref{eq:mu}:
\bbe\label{eq:chi22}
\chi(\eta)({\bf x})=\chi(\eta)(x_1),\quad x_1=\mu({\bf x})\in K,\quad {\bf x}\in G_A.
\ee 
We have a diagonal embedding of $K$ into $\bA_{K,\la}$ via $\xi\mapsto(\xi,\ldots,\xi)$. We will denote by $\xi$ a general element of $K$, by ${\bf x}$ a general element of $G_A$, and by $x_1$ a general element of $\mu(G_A)$. In what follows, we will show that $\widehat{G_A}$ is isomorphic to a quotient of $\bA_{K,\la}$ by $Y_{A}({\bf u},\la)$ defined by \eqref{eq:ya}. 
\begin{thm}\label{th:onto}
Let $A\in\M_n(\bZ)$ be non-singular with an irreducible characteristic polynomial $h_A\in\bZ[x]$.
Let $\la\in\overline{\bQ}$ be an eigenvalue of $A$ and let $K=\bQ(\la)$. 
The map $\phi:\bA_{K,\la}\to \widehat{G_A}$ given by $(\ldots,\eta_{\fp},\ldots)\mapsto \chi$,
\bbe\label{eq:charmain}
\chi({\bf x})=\prod_{\fp\in \cS_{\infty}\cup\cS_{\la}} e^{2\pi i\La_{\fp}(\eta_{\fp}\mu({\bf x}))},\quad{\bf x}\in G_A,
\ee
is onto. Moreover, $\ker\phi=Y_{A}({\bf u},\la)$, so that 
$$
\widehat{G_A}\cong \bA_{K,\la}/Y_{A}({\bf u},\la).
$$
\end{thm}
We divide the proof of Theorem \ref{th:onto} into parts proved in Lemmas \ref{l:onto} -- \ref{l:inside} below. Namely, $\phi$ is onto (Lemma \ref{l:onto}), $Y_{A}({\bf u},\la)\subseteq\ker\phi$ (Lemma \ref{l:incl}), if $\eta\in\bA_{K,\la}$ is trivial on $\cO_{K,\la}$, then $\eta\in K$ (Lemma \ref{l:constant}), and 
$\ker\phi \subseteq Y_{A}({\bf u},\la)$ (Lemma \ref{l:inside}).
\begin{lem}\label{l:onto}
$\phi$ is onto.
\end{lem}
\begin{proof}[Proof of Lemma $\ref{l:onto}$]
Let $\cS$ denote the set of {\it all} finite places of $K$ dividing all $p\in\cP$, {\it i.e.}, {\it all} prime ideals of $\cO_K$ lying above all $p\in\cP$ as opposed to $\cS_{\la}$, the set of prime ideals of $\cO_K$ dividing $\la$. Let
$\bA_{K,\cS}=\prod_{\fp\in\cS_{\infty}\cup\cS}K_{\fp}$. Similar to \eqref{eq:charmain}, we also have a natural homomorphism
$\psi:\bA_{K,\cS}\to \widehat{Y_{A^t}({\bf w},\la)}\cong\widehat{G_A}$ given by $(\ldots,\eta_{\fp},\ldots)\mapsto \chi$,
\bbe\label{eq:psii}
\chi(x_1)=\prod_{\fp\in \cS_{\infty}\cup\cS} e^{2\pi i\La_{\fp}(\eta_{\fp}x_1)},\quad{x}_1\in Y_{A^t}({\bf w},\la)
\ee
(see Lemma \ref{l:image}). 
We now show that $\psi$ is onto. Indeed, 
by Lemma \ref{l:grchargen}, for any $\chi\in\widehat{G_A}$  there exist ${\bf\theta}\in\bR^n$ and ${\bf v}_p\in\bZ_p^n$, $p\in\cP$, such that 
\bbe\label{eq:xxxx}
\chi({\bf x})=e^{-2\pi i{\bf x}\cdot{\bf \theta}} \prod_{p\in\cP}e^{2\pi i\{{\bf x}\cdot{\bf v}_p\}_p},\quad {\bf x}\in G_A.
\ee
Let $n=r_1+2r_2$, where $r_1$ is the number of real roots of $h_A$ and $r_2$ is the number of pairs of conjugate complex roots of $h_A$. Without loss of generality, we can assume that $\cS_{\infty}=\{\sg_1,\ldots,\sg_{r_1},\sg_{r_1+1},\ldots,\sg_{r_1+r_2}\}$. For ${\bf\theta}\in\bR^n$, denote
$\eta_{\fp}=\sg_j({\bf u})\cdot{\bf\theta}\in K_{\fp}$, where $\fp\in\cS_{\infty}$ corresponds to $\sg_j$, $j\in\{1,\ldots,r_1+r_2\}$.  Then, 
\bbe\label{eq:infty1}
e^{-2\pi i {\bf x}\cdot {\bf\theta}}=\prod_{\fp\,\in\,\cS_{\infty}}e^{2\pi i\La_{\fp}(\eta_{\fp}x_1)},\quad
x_1=\mu({\bf x}).
\ee
Recall that ${\bf x}=\sum_{i=1}^{n}\sg_i(x_1{\bf u})$. Then 
\bbe\label{eq:xvp}
\{{\bf x}\cdot{\bf v}_p\}_p=\sum_{\fp\vert p}\{\tr_{K_{\fp}/\bQ_p}({\bf v}_p\cdot {\bf u}x_1)\}_p=
\sum_{\fp\vert p}\La_{\fp}(\eta_{\fp}x_1),\quad \eta_{\fp}={\bf v}_p\cdot {\bf u},
\ee
which together with \eqref{eq:infty1}, shows that $\psi$ given by \eqref{eq:psii} is onto. Denote
$$
\caR[{\bf u}]=\Span_{\caR}(u_1,\ldots,u_n),\,\,
\bZ[{\bf u}]=\Span_{\bZ}(u_1,\ldots,u_n),\,\, \bZ_p[{\bf u}]=\bZ[{\bf u}]\otimes_{\bZ}\bZ_p,
$$
where $\caR[{\bf u}]\subset K$, $\bZ[{\bf u}]\subset \cO_K$, and $\bZ_p[{\bf u}]\subset \cO_{\fp}$.
We claim that $\caR[{\bf u}]\subseteq \ker\psi$, where $\caR[{\bf u}]$ is embedded diagonally into $\bA_{K,\cS}$. Indeed, if
$\eta={\bf r}\cdot{\bf u}\in\caR[{\bf u}]$, ${\bf r}\in\caR^n$, then
$$
\sum_{\fp\in \cS_{\infty}\cup\cS}\La_{\fp}(\eta x_1)=-{\bf r}\cdot{\bf x}+\sum_{p\in\cP}\{{\bf r}\cdot{\bf x}\}_p=0
$$
by \cite[Lemma 4.1.5]{t}, since ${\bf r}\cdot{\bf x}\in\caR$ ({\it i.e.}, 
${\bf r}\cdot{\bf x}\in\bQ$ and has only primes from $\cP$ in the denominators) for any ${\bf x}\in G_A$. 
Hence,  
$$
(\ldots,\eta,\ldots)\mapsto\chi(x_1)=\prod_{\fp\in \cS_{\infty}\cup\cS} e^{2\pi i\La_{\fp}(\eta x_1)}=1
$$
for any $x_1\in Y_{A^t}({\bf w},\la)$.
One can show that any element from $\bA_{K,\cS}$ is equivalent to an element $\eta=(\ldots,\eta_{\fp},\ldots)$ from 
$\Om_1=\prod_{\fp\in\cS_{\infty}}K_{\fp}\prod_{p\in\cP}\prod_{\fp\vert p}\bZ_p[{\bf u}]$ modulo $\caR[{\bf u}]$, where for each $p\in\cP$ there exists $\eta_p\in\bZ_p[{\bf u}]$ such that
$\eta_{\fp}=\eta_p$ for any $\fp\vert p$. It also follows from \eqref{eq:xvp}.

\sbr

We now consider the restriction of $\psi$ from $\bA_{K,\cS}$ to $\bA_{K,{\la}}$ and show that the restriction is also onto $\widehat{Y_{A^t}({\bf w},\la)}\cong\widehat{G_A}$. This is true, because even though $x_1$ might not be in $\cO_{\fq}$ for any prime ideal $\fq$ not in $\cS_{\la}$, but $\val_{\fq}x_1$ is bounded from below by a constant that does not depend on $x_1$. Indeed, by the previous paragraph, without loss of generality, we can assume that $\chi\in \widehat{Y_{A^t}({\bf w},\la)}$ is defined by $\eta\in\Om_1$. Let $L$ be a finite Galois extension of $\bQ$ containing all eigenvalues of $A$, {\it e.g.}, $L$ is the splitting field of $h_A$. 
Recall that $x_1=\mu({\bf x})=a\la^{-k}(\det M)^{-1}$, where $a\in\cO_L$, $k\in\bN\cup\{0\}$, $\det M\ne 0\in\cO_L$ by Lemma \ref{l:image}. For any prime ideal $\fq$ of $\cO_L$ not dividing $\la$ lying above a prime $q\in\bN$, there exists $k_{\fq}\in\bN\cup\{0\}$ such that ${q}^{k_{\fq}}(\det M)^{-1}\in\cO_{\fq}$. Since there are finitely many prime ideals $\fq$ of 
$\cO_L$ lying above $q$, by taking the maximum among all $k_{\fq}$, we can assume that there exists $k_q\in\bN\cup\{0\}$ such that $q^{k_{q}}(\det M)^{-1}\in\cO_{\fq}$ for any $\fq$ above $q$. Then, 
$p^{k_p}x_1\in\cO_{\fp}$ for any $x_1$ and any prime ideal $\fp$ of $K$ above $p\in\cP$ not dividing $\la$, {\it i.e.}, 
$\fp\not\in \cS_{\la}$. We now write each $\eta_p\in\bZ_p[{\bf u}]$ as $\eta_p=a_p+p^{k_p}\mu_p$ for $a_p\in\bZ[{\bf u}]$
and $\mu_p\in\bZ_p[{\bf u}]$. By the Chinese Remainder Theorem, there exists $a\in\bZ[{\bf u}]$ such that 
$\eta_p-a\in p^{k_p}\bZ_p[{\bf u}]$ for any $p\in\cP$. Then $\eta-a$ defines a character of $\widehat{Y_{A^t}({\bf w},\la)}$ as follows:
\bbe\label{eq:leftright}
(\ldots,\eta_p-a,\ldots)\mapsto\chi(x_1)=\prod_{\fp\in \cS_{\infty}\cup\cS} e^{2\pi i\La_{\fp}((\eta_p-a) x_1)}=
\prod_{\fp\in \cS_{\infty}\cup\cS_{\la}} e^{2\pi i\La_{\fp}((\eta_p-a) x_1)}.
\ee
Indeed, $\La_{\fp}((\eta_p-a) x_1)=0$ for any $\fp$ not in $\cS_{\la}$ above a prime $p\in\cP$,
since $(\eta_p-a)x_1\in\cO_{\fp}$.  
This shows that every character of $Y_{A^t}({\bf w},\la)$, equivalently, of $G_A$, comes from an element from 
$\bA_{K,{\la}}$, hence $\phi$ is onto.
\end{proof}

\begin{lem}\label{l:incl}
$Y_{A}({\bf u},\la)\subseteq\ker\phi$.
\end{lem}

\begin{proof}
There exists $k\in\bN$ such that $\la^kx_1\in\cO_{\fp}$ for any $x_1\in\bZ[{\bf w}]$ and any $\fp\in\cS_{\la}$. Since multiplication by a power of $\la$ is an isomorphism of $Y_{A^t}({\bf w},\la)$, by precomposing every character of 
$Y_{A^t}({\bf w},\la)$ with the isomorphism, without loss of generality, we can assume that $x_1$ itself lies in $\cO_{\fp}$, {\it i.e.}, 
$\bZ[{\bf w}]\subset \cO_{\fp}$ for any $\fp\in\cS_{\la}$. Also, since $\ker\phi$ is a $\bZ[\la^{\pm 1}]$-module, it is enough to show that $\bZ[{\bf u}]\subseteq\ker\phi$. Let $\eta={\bf s}\cdot{\bf u}\in\bZ[{\bf u}]$, ${\bf s}\in\bZ^n$. 
By the proof of the previous lemma, 
$\caR[{\bf u}]\subseteq \ker\psi$. Thus, for $\eta\in\bZ[{\bf u}]$ we have that 
\bbe\label{eq:help}
0=\sum_{\fp\in\cS_{\infty}\cup\cS}\La_{\fp}(\eta x_1)=\sum_{\fp\in\cS_{\infty}\cup\cS_{\la}}\La_{\fp}(\eta x_1)+
\sum_{p\in\cP}\sum_{\fp\vert p,\, \fp\not\in\cS_{\la}}\La_{\fp}(\eta x_1),\quad\forall x_1\in Y_{A^t}({\bf w},\la).
\ee
Denote $T_p=\sum_{\fp\vert p,\, \fp\not\in\cS_{\la}}\La_{\fp}$. From \eqref{eq:help}, it is enough to show that 
$T_p(\eta x_1)=0$ for any $x_1\in Y_{A^t}({\bf w},\la)$ and $p\in\cP$.
We have that 
\bbe\label{eq:summm}
\{{\bf s}\cdot{\bf x}\}_p=\sum_{\fp\vert p}\La_{\fp}(\eta x_1)=
\sum_{\fp\vert p,\,\fp\in\cS_{\la}}\La_{\fp}(\eta x_1)+T_p(\eta x_1),\quad\forall x_1\in Y_{A^t}({\bf w},\la).
\ee
Since $x_1\in\bZ[{\bf w}]$ if and only if ${\bf x}\in\bZ^n$, we have that $\{{\bf s}\cdot{\bf x}\}_p=0$ for any $x_1\in\bZ[{\bf w}]$. In addition, $\La_{\fp}(\eta x_1)=0$ for any $x_1\in\bZ[{\bf w}]$ and any $\fp\in\cS_{\la}$, as follows from our assumption. Therefore, from \eqref{eq:summm}, $T_p(\eta x_1)=0$ for any $x_1\in\bZ[{\bf w}]$. Since 
$\la$ is a unit in the ring of integers of $K_{\fp}$ for any $\fp$ that does not divide $\la$, multiplication by $\la$ is a $\bZ_p$-module automorphism of $\bZ_p[{\bf w}]$, hence $T_p(\eta x_1)=0$ for any $x_1\in Y_{A^t}({\bf w},\la)$.
 \end{proof}
 
 \begin{lem}\label{l:constant}
 If $\eta\in\bA_{K,\la}$ is trivial on $\cO_{K,\la}$, then $\eta\in K$. 
\end{lem}

\begin{proof}
Let $\eta=(\ldots,\eta_{\fp},\ldots)\in\ker\phi$, {\it i.e.}, 
\bbe\label{eq:kernel}
\sum_{\fp\in\cS_{\infty}\cup\cS_{\la}}\La_{\fp}(\eta_{\fp} x_1)=0,\quad\forall x_1\in Y_{A^t}({\bf w},\la).
\ee
Note that \eqref{eq:kernel} holds for any $x_1\in\cO_{K,\la}$, since $\cO_{K,\la}\subseteq Y_{A^t}({\bf w},\la)$ by Lemma \ref{l:mu1} and Lemma \ref{l:image}. By the Chinese Remainder Theorem, there exists 
$a\in \cO_{K,\la}$ such that $\eta_{\fp}-a\in\cO_{\fp}$ for all $\fp\in\cS_{\la}$. Let $\xi=\eta-a$ with $\xi_{\fp}\in\cO_{\fp}$
for all $\fp\in\cS_{\la}$. Then
\bbe\label{eq:kk}
\sum_{\fp\in\cS_{\infty}\cup\cS_{\la}}\La_{\fp}(\xi_{\fp} x_1)=0,\quad\forall x_1\in \cO_{K,\la},
\ee
since $a\in\cO_{K,\la}$ defines a trivial character on $\cO_{K,\la}$. Moreover, $\La_{\fp}(\xi_{\fp}x_1)=0$ for any $x_1\in\cO_K$ and $\fp\in\cS_{\la}$, hence 
\bbe\label{eq:infty}
\sum_{\fp\in\cS_{\infty}}\La_{\fp}(\xi_{\fp} x_1)=0,\quad\forall x_1\in\cO_K.
\ee
As in \cite{t}, we denote by $\overset{\infty}{\xi}$ the projection of $\xi$ onto $\prod_{\fp\in\cS_{\infty}}K_{\fp}$. Since $u_1,\ldots,u_n$ is a basis of $K$ as a $\bQ$-vector space, $\overset{\infty}{u_1},\ldots,\overset{\infty}{u_n}$ is a basis of 
$\prod_{\fp\in\cS_{\infty}}K_{\fp}$ as an $n$-dimensional $\bR$-vector space. Let $\overset{\infty}{\xi}=\sum_{i=1}^n\al_i\overset{\infty}{u_i}$ for some $\al_1,\ldots,\al_n\in\bR$. Applying \eqref{eq:infty} to each $x_1=u_i\in\cO_K$, we get that
$MM^t{\bf\al}\in\bZ^n$ for ${\bf\al}=\left(\begin{matrix}\al_1 &\ldots & \al_n \end{matrix}\right)^t$ with $M$ given by \eqref{eq:lm}. One can check that $MM^t\in\M_n(\bZ)$ and hence $\al_1,\ldots,\al_n\in\bQ$. Thus, 
$\overset{\infty}{\xi}=(b,\ldots,b)$ for $b=\sum_{i=1}^n\al_iu_i$, $b\in K$. We now show that $\xi=b$, equivalently, 
$\xi_{\fp}=0$ for any $\fp\in\cS_{\la}$, so that $\eta=\xi+a=b+a\in K$. Indeed,  
there exists $l\in\bN$ such that $lb\in \cO_K$. Denote $\kappa=l(\xi-b)$, so that $\overset{\infty}{\kappa}=(0,\ldots,0)$,  
$\kappa_{\fp}\in\cO_{\fp}$ for any $\fp\in\cS_{\la}$, and 
$$
\sum_{\fp\in\cS_{\la}}\La_{\fp}(\ka_{\fp} x_1)=0,\quad\forall x_1\in \cO_{K,\la}.
$$
Then, $\sum_{\fp\vert p}\La_{\fp}(\ka_{\fp} x_1)=0$, equivalently, $\sum_{\fp\vert p}\tr_{K_{\fp}/\bQ_p}(\ka_{\fp} x_1)\in\bZ_p$ for any $x_1\in \cO_{K,\la}$ and any $p\in\cP$. Since the image of $\cO_{K,\la}$ under the embedding $K\hrar K_{\fp}$ generates $K_{\fp}$ over $\bZ_p$, we have that 
$\sum_{\fp\vert p}\tr_{K_{\fp}/\bQ_p}(\ka_{\fp} u_i)=0$ for any $p\in\cP$ and $i\in\{1,\ldots,n\}$. It gives a system of linear equations for $\kappa_{\fp}$ with matrix $M^t$, which is non-singular. Therefore, each $\kappa_{\fp}=0$.
\end{proof}

\begin{lem}\label{l:inside}
$\ker\phi \subseteq Y_{A}({\bf u},\la)$.
\end{lem}

\begin{proof}
Let $\eta\in\ker\phi$. By Lemma \ref{l:mu1} and
Lemma \ref{l:image},  $\cO_{K,\la}\subseteq\mu(G_A)$, $\mu(G_A)=Y_{A^t}({\bf w},\la)$, so that $\eta$ is trivial on $\cO_{K,\la}$. Hence, 
by Lemma \ref{l:constant}, $\eta\in K$ embedded diagonally into $\bA_{K,\la}$.
By \cite[Lemma 4.1.5]{t}, we have that 
\bbe\label{eq:product1}
\sum_{\fp\in\cS_{\infty}\cup\cS_{\la}}\La_{\fp}(\eta x_1)+\sum_{p\in\cP}\sum_{\fp\vert p,\,\fp\not\in\cS_{\la}}\La_{\fp}(\eta x_1)+\sum_{q\not\in\cP,\,\fq\vert q}\La_{\fq}(\eta x_1)=0,\quad\forall x_1\in Y_{A^t}({\bf w},\la),
\ee
where $\fp,\fq$ are prime ideals of $\cO_K$, and $p,q\in\bN$ are prime numbers. Since $\eta\in\ker\phi$, we have that 
$\sum_{\fp\in\cS_{\infty}\cup\cS_{\la}}\La_{\fp}(\eta x_1)=0$ for any $x_1\in Y_{A^t}({\bf w},\la)$. Then from \eqref{eq:product1}, for any $p\in\cP$ and any $q\not\in\cP$ we have that
$$
T_p(\eta x_1)\equiv\sum_{\fp\vert p,\,\fp\not\in\cS_{\la}}\La_{\fp}(\eta x_1)=0,\quad 
T_q(\eta x_1)\equiv\sum_{\fq\vert q}\La_{\fq}(\eta x_1)=0,\quad\forall{x_1\in Y_{A^t}({\bf w},\la)}.
$$
Since $h_A$ is irreducible by assumption, $K=\Span_{\bQ}(u_1,\ldots,u_n)$. Let $\eta={\bf t}\cdot{\bf u}$, ${\bf t}\in\bQ^n$.
Then $T_q(\eta x_1)=\{\tr_{K/\bQ}(\eta x_1)\}_q=\{{\bf t}\cdot{\bf x}\}_q=0$ for any ${\bf x}\in G_A$. Since $\bZ^n\subseteq G_A$, this implies that ${\bf t}\in\caR^n$. Note that there exists $k\in\bN\cup\{0\}$ such that
\bbe\label{eq:integral}
\la^k\eta x_1\in\cO_{\fp},\quad\forall x_1\in\bZ[{\bf w}],\,\forall\fp\in\cS_{\la}.
\ee
Also, note that $T_p(\la^k \eta x_1)=0$ for any $x_1\in Y_{A^t}({\bf w},\la)$, since $Y_{A^t}({\bf w},\la)$ is a $\bZ[\la^{\pm 1}]$-module. Let $\la^k\eta={\bf s}\cdot{\bf u}$ for some ${\bf s}\in\bQ^n$. 
Note that $x_1\in\bZ[{\bf w}]$ if and only if ${\bf x}\in\bZ^n$. We have that 
$$
\{{\bf s}\cdot{\bf x}\}_p=\sum_{\fp\vert p}\La_{\fp}(\la^k\eta x_1)=
\sum_{\fp\vert p,\,\fp\in\cS_{\la}}\La_{\fp}(\la^k\eta x_1)+T_p(\la^k\eta x_1),\quad\forall x_1\in Y_{A^t}({\bf w},\la),
$$
and therefore, $\{{\bf s}\cdot{\bf x}\}_p=0$ for ${\bf x}\in\bZ^n$. 
This implies that ${\bf s}\in(\bZ_{(p)})^n$ for any $p\in\cP$, where $\bZ_{(p)}$ consists of all rational numbers $a/b\in\bQ$ such that $(b,p)=1$. We now have that 
$\la^k\eta={\bf s}\cdot {\bf u}$ and $\eta={\bf t}\cdot {\bf u}$, hence ${\bf s}=(A^k)^t{\bf t}$, since $u_1,\ldots,u_n$ are linearly independent over $\bQ$. Thus, ${\bf s}\in(\caR\cap_{p\in\cP}\bZ_{(p)})^n$ and therefore ${\bf s}\in\bZ^n$ and
$\la^k\eta\in\bZ[{\bf u}]$. Hence, $\eta\in\la^{-k}\bZ[{\bf u}]\subset Y_{A}({\bf u},\la)$. This shows that 
$\ker\phi\subseteq Y_{A}({\bf u},\la)$.
\end{proof}

We now find the fundamental domain $\cF$ of the action of $\Ga=\ker\phi$ on $\bA_{K,\la}$. We will use the result in Section \ref{ss:fixed} below to count the number of periodic points of a continuous endomorphism of toroidal solenoid $\bS_{A}$. As in the proof of Lemma \ref{l:onto}, every element
 of $\bA_{K,\la}$ is equivalent modulo $\Ga$ to an element of 
$$
\Om=\prod_{\fp\in\cS_{\infty}}K_{\fp}\times\prod_{p\in\cP}\prod_{\fp\in\cS_{\la},\,\fp\vert p}\bZ_p
[{\bf u}],
$$
so that $\cF\subseteq\Om$. Moreover, by Lemma \ref{l:incl}, $\bZ[{\bf u}]\subseteq\ker\phi$ and hence
\bbe\label{eq:funddom}
\cF=[0,1)^n\times\prod_{p\in\cP}\prod_{\fp\in\cS_{\la},\,\fp\vert p}\bZ_p
[{\bf u}],
\ee
where $\prod_{\fp\in\cS_{\infty}}K_{\fp}$ is considered as an $n$-dimensional $\bR$-vector space with respect to the basis
$\overset{\infty}{u_1},\ldots,\overset{\infty}{u_n}$.  
Since $\cF$ has an interior, $\Ga$ is discrete. 

\section{Toroidal solenoids}\label{s:toroidal}
In this section, 
we apply our results concerning groups $G_A$ and their endomorphisms to the case of toroidal solenoids. 
Let $\bT^n$ denote a torus considered as a quotient of  
$\bR^n$ by its subgroup $\bZ^n$. A matrix $A\in\operatorname{M}_n(\bZ)$ induces a map $A:\bT^n\rar \bT^n$,
$A\left([{\bf x}]\right)=[A{\bf x}]$, $[{\bf x}]\in\bT^n$, ${\bf x}\in\bR^n$. Consider the inverse system $(M_j,f_j)_{j\in\bN}$, where $f_j:M_{j+1}\rar M_{j}$,  
$M_j=\bT^n$ and $f_j=A$ for all $j\in\bN$.
The inverse limit ${\bS}_A$ of the system is called a  ({\em toroidal}) {\em solenoid}. As a set, $\bS_A$ is a subset of 
$\prod_{j=1}^{\infty} M_j$ consisting of points $(z_j)\in\prod_{j=1}^{\infty} M_j$ such that $z_j\in M_j$ and $f_j(z_{j+1})=z_{j}$ for $\forall j\in\bN$, {\it i.e.},
\bbe\label{eq:solen}
\bS_A=\left\{(z_j)\in\prod_{j=1}^{\infty}\bT^n\,\,\Big\vert\,\,z_j\in\bT^n,\,\, A(z_{j+1})=z_j,\,\,j\in\bN \right\}.
\ee
Endowed with the natural group structure and the induced topology from the Tychonoff (product) topology on $\prod_{j=1}^{\infty} \bT^n$, $\bS_A$ is an  
$n$-dimensional topological abelian group. It is compact, metrizable, and connected, but not locally connected and not path connected. The map $\sg:\bS_A\to\bS_A$ induced by multiplication by $A$, $(z_j)\mapsto (A(z_j))$ is an automorphism of $\bS_A$ as a topological group, and the pair $(\bS_A,\sg)$ is a dynamical system. 
The endomorphism ring $\End(\bS_A,\sg)$ of the dynamical system $(\bS_A,\sg)$ consists of 
endomorphisms of $\bS_A$ as a topological group that commute with $\sg$. It is known that 
$\widehat{G_{A^t}}\cong\bS_A$. Indeed, $G_{A^t}$ is a direct limit of groups $(A^t)^{-j}\bZ^n$, $j\in\bN\cup\{0\}$. Here, each $(A^t)^{-j}\bZ^n$ is isomorphic to $\bZ^n$ and the maps $(A^t)^{-j}\bZ^n\to (A^t)^{-(j+1)}\bZ^n$ are inclusions. Applying the functor $\Hom_{\bZ}(-,\bT^1)$ to the system, we obtain the inverse limit of groups
$$\Hom_{\bZ}((A^t)^{-j}\bZ^n,\bT^1)\cong \Hom_{\bZ}(\bZ^n,\bT^1)\cong\bT^n$$
with the maps $f_j$ as above that defines $\bS_A$. This gives an isomorphism of topological groups $\widehat{G_{A^t}}\cong\bS_A$, where $G_{A^t}$ is endowed with the discrete topology and $\widehat{G_{A^t}}$ is endowed with the compact-open topology. Moreover, since $G_{A^t}$ is a locally compact abelian group, it follows from the Pontryagin duality theorem that the map between the rings $\End(G_{A^t})$ and $\End(\widehat{G_{A^t}})$ given by 
$\phi\mapsto (\chi\mapsto \chi\circ\phi)$, $\phi\in \End(G_{A^t})$, $\chi\in\widehat{G_{A^t}}$, 
 is a ring isomorphism between the opposite ring $\End(G_{A^t})^{op}$ of $\End(G_{A^t})$ and 
 $\End(\widehat{G_{A^t}})$. Thus, we have a ring isomorphism $\End(\bS_A)\cong\End(G_{A^t})^{op}$, under which 
 $\sg$ corresponds to multiplication by $A^t$ on $G_{A^t}$. Therefore, $\End(\bS_A,\sg)$ is isomorphic to the subring of $\End(G_{A^t})^{op}$ consisting of $T\in\M_n(\bQ)\cap \End(G_{A^t})$ that commute with $A^t$. By Corollary \ref{cor:abel}, under the assumptions of Proposition \ref{pr:end}, $\End(G_{A^t})$ is commutative and, in particular, every endomorphism of $G_{A^t}$ commutes with $A^t$. This implies that every endomorphism of $\bS_A$ commutes with 
 $\sg$ and hence is an endomorphism of the dynamical system $(\bS_A,\sg)$. 
 Thus, we have the following

\begin{prop}\label{pr:endtor} 
Assume $A\in\M_n(\bZ)$ is non-singular with an irreducible characteristic polynomial, $\cP'\ne\emptyset$, and  
 there exists a prime $p\in\bN$ with 
$(n,t_p)=1$. Then $\End(\bS_A)$ is a commutative ring and, in particular, $\End(\bS_A,\sg)=\End(\bS_A)$.
\end{prop}

\subsection{$\cS$-integer dynamical systems}
In \cite{cew}, the authors introduce the so-called $\cS$-integer dynamical system $(X,\al)$, a dual object to 
the group of $\cS$-integers $\cO_{K,\cS}$ in a number field $K$ and an element $\xi\in \cO_{K,\cS}$. 
Groups of the form $G_A$ arise topologically as character groups of toroidal solenoids. Thanks to the description of endomorphisms of $G_A$ in Proposition \ref{pr:end} and the description of $G_A$ as a subset of a number field in Lemma \ref{l:image}, one can see a connection between toroidal solenoids and $\cS$-integer dynamical systems. 

\begin{defin}\label{def:sin}
Let $K$ be a number field and let $\cS$ be a set of prime ideals of the ring of integers $\cO_K$ of $K$. Let $\xi\in K$, $\xi\ne 0$, $\xi\in\cO_{K,\cS}$, where $\cO_{K,\cS}$ is the ring of $\cS$-integers of $K$ defined as
$$
\cO_{K,\cS}=\{x\in K\,\vert\, \val_{\fp}(x)\geq 0\text{ for any prime ideal }\fp\text{ of }\cO_K\text{ not in }\cS\}.
$$
Let $X\cong\widehat{\cO_{K,\cS}}$ be the (Pontryagin) dual to the discrete (countable) group $\cO_{K,\cS}$ and 
let $\al:X\to X$ be a continuous group endomorphism, the dual to the monomorphism $\hat\al:\cO_{K,\cS}\to \cO_{K,\cS}$ given by 
 $\hat\al(x)=\xi x$. A pair $(X,\al)=(X^{(K,\cS)},\al^{(K,\cS,\xi)})$ is called
 an $\cS$-{\em integer dynamical system} \cite{cew}.
\end{defin}

Assume $A\in\M_n(\bZ)$ is non-singular with an irreducible characteristic polynomial 
and let $\la\in\overline{\bQ}$ be an eigenvalue of $A$, $K=\bQ(\la)$. It is known that $\bS_A\cong \widehat{G_{A^t}}$ as topological groups, where $G_{A^t}$ is endowed with the discrete topology. By \eqref{eq:diagram}, we know that 
$G_{A^t}\cong Y_{A}({\bf u},\la)$, where ${\bf u}=\left(\begin{matrix} u_1 & \ldots & u_n\end{matrix} \right)\in(\cO_K)^n$ is an eigenvector of $A$ corresponding to $\la$, and 
$Y_{A}({\bf u},\la)$ is a $\bZ[\la^{\pm 1}]$-submodule of $\cO_{K,\la}$ given by
\bbe\label{eq:ya1}
Y_{A}({\bf u},\la)=\{m_1\la^{k_1}u_1+\cdots +m_n\la^{k_n}u_n\,\vert\,m_1,\ldots,m_n,k_1,\ldots,k_n\in\bZ \}.
\ee
%
This implies that $\bS_A\cong \widehat{Y_{A}({\bf u},\la)}$,  a description of $\bS_A$ in the spirit of Definition \ref{def:sin}.
Assume in addition that $\cP'\ne\emptyset$ and that there exists a prime $p\in\bN$ with 
$(n,t_p)=1$. Then 
under 
the isomorphism 
$G_{A^t}\cong Y_{A}({\bf u},\la)$ given by \eqref{eq:mu} applied to $A^t$, an endomorphism $\al$ of $\bS_A$ is dual to an endomorphism $\hat\al(y)=xy$, $y\in Y_{A}({\bf u},\la)$, where $x=\imath(T)$, $T\in\End(G_{A^t})$, via \eqref{eq:tx} and the proof of Proposition  \ref{pr:end}. Recall that $\cO_{K,\la}=\cO_K[\la^{-1}]$. It is a ring of $\cS$-integers with $\cS$ consisting of prime ideals of $\cO_K$ dividing $\la$. Thus, in general, $\bS_A$ is not an $\cS$-integer dynamical system, since it corresponds to a {\it subring} $Y_{A}({\bf u},\la)$ of a ring of $\cS$-integers $\cO_{K,\la}$. However, we show below that  similar results can be proved regarding toroidal solenoids. 
This suggests that there might be a more general object that encompasses both $\cS$-integer dynamical systems and toroidal solenoids corresponding to matrices with irreducible characteristic polynomials satisfying the condition $(n,t_p)=1$.

\begin{thm}\label{th:sintsol}
Assume $A\in\M_n(\bZ)$ is non-singular with an irreducible characteristic polynomial 
and $\cP'\ne\emptyset$. 
Assume, in addition, that there exists a prime $p\in\bN$ with 
$(n,t_p)=1$. Let $\la\in\overline{\bQ}$ be an eigenvalue of $A$. Then $\bS_{A}$ is isomorphic to $\widehat{Y}$, where $Y$ is a $\bZ[\la^{\pm 1}]$-submodule of $\overline{\bQ}$ generated by $u_1,\ldots, u_n$, where
${\bf u}=\left(\begin{matrix} u_1 & \ldots & u_n \end{matrix}\right)\in(\overline{\bQ})^n$ is an eigenvector of $A$ corresponding to $\la$, and $Y$ is endowed with the discreet topology. If $K=\bQ(\la)$ and ${\bf u}\in(\cO_K)^n$, then
$Y\subseteq\cO_{K,{\la}}$, and
$$
\Span_{\bZ}(u_1,\ldots, u_n)\subseteq Y\subseteq\Span_{\caR}(u_1,\ldots, u_n),
$$
where $\caR=\bZ\left[ \frac{1}{\det A}\right]$. Moreover, under the isomorphism, each endomorphism 
$\al$ of $\bS_A$ is dual to an endomorphism $\hat\al(y)=xy$, $y\in Y$, $x\in\cO_{K,{\la}}$, of $Y$. 
\end{thm}

\begin{cor}
Assume $A\in\M_n(\bZ)$ is non-singular with an irreducible characteristic polynomial 
and $\cP'\ne\emptyset$. 
Assume, in addition, that there exists a prime $p\in\bN$ with 
$(n,t_p)=1$.  If there exist an eigenvalue $\la\in\overline{\bQ}$ with a corresponding eigenvector 
${\bf u}=\left(\begin{matrix} u_1 & \ldots & u_n \end{matrix}\right)\in(\cO_K)^n$ of $A$, $K=\bQ(\la)$, such that 
$\Span_{\bZ}(u_1,\ldots, u_n)=\bZ[\la]=\cO_K$, then 
$(\bS_A,\al)$, $\al\in\End(\bS_A)$, is an $\cS_{\la}$-integer dynamical system with $\cS_{\la}$ consisting of prime ideals of $\cO_K$ dividing $\la$.
\end{cor}
\subsection{Ergodicity and periodic points of endomorphisms of $\bS_A$}\label{ss:fixed}
We now apply the characterization of $\bS_A$ via Theorem \ref{th:onto} to count numbers of periodic 
points of a continuous endomorphism $T\in\End(\bS_{A})$ of $\bS_A$. 
\begin{prop}\label{pr:charsol1}
Assume $A\in\M_n(\bZ)$ is non-singular with an irreducible characteristic polynomial. 
Let $\la\in\overline{\bQ}$ be an eigenvalue of $A$, $K=\bQ(\la)$. Let 
${\bf w}=\left(\begin{matrix} w_1 & \ldots & w_n \end{matrix}\right)\in(\cO_K)^n$ be an eigenvector of $A^t$ corresponding to $\la$.  
We have an isomorphism
$$
{\bS_A}\cong \bA_{K,\la}/Y_{A^t}({\bf w},\la),
$$
where 
$$
Y_{A^t}({\bf w},\la)=\{m_1\la^{k_1}w_1+\cdots +m_n\la^{k_n}w_n\,\vert\,m_1,\ldots,m_n,k_1,\ldots,k_n\in\bZ \},
$$
and $Y_{A^t}({\bf w},\la)$ 
is embedded diagonally into $\bA_{K,\la}$. 
\end{prop}
\begin{proof}
Since ${\bS_A}\cong \widehat{G_{A^t}}$ (see the $1$st paragraph  of Section \ref{s:toroidal}), the proposition follows from 
Theorem \ref{th:onto} applied to $A^t$. 
\end{proof}

Let $T\in\End(\bS_{A})$ be a continuous endomorphism  of $\bS_A$. 
By definition, 
the set $F_k(T)$ of points of period $k\geq 1$ of $T$ is 
$$ 
F_k(T)=\{x\in \bS_{A}\,\vert \, T^k(x)=x\}.
$$

\begin{lem}\label{l:erg} Assume $A\in\M_n(\bZ)$ is non-singular with an irreducible characteristic polynomial 
and $\cP'\ne\emptyset$. 
Assume, in addition, that there exists a prime $p\in\bN$ with 
$(n,t_p)=1$. Then 
$T\in\End(\bS_A)$ is ergodic if and only if  each eigenvalue of $T$ is not a root of unity.
\end{lem}

\begin{proof}
It follows from \cite[Theorem 4.2]{cew}, Proposition \ref{pr:end}, and Theorem \ref{th:sintsol}.
\end{proof}

Let $T:\bS_A\rar\bS_A$ be a continuous homomorphism. 
Since $\bS_A\cong\widehat{G_{A^t}}$ as topological groups, $T$ induces a homomorphism (denoted by the same letter by abuse of notation) $T:\widehat{G_{A^t}}\rar \widehat{G_{A^t}}$.  Since
$G_{A^t}$ is locally compact with respect to the discrete topology, we have that  $\widehat{\widehat{G_{A^t}}}\cong G_{A^t}$ and the induced dual $T:G_{A^t}\rar G_{A^t}$ is a homomorphism. By Proposition \ref{pr:end}, 
$\xi=\imath(T)\in\cO_{K,{\la}}$. 

\begin{prop}\label{pr:period}
Assume $A\in\M_n(\bZ)$ is non-singular with an irreducible characteristic polynomial 
and 
$\cP'\ne\emptyset$. 
Assume, in addition, that there exists a prime $p\in\bN$ with 
$(n,t_p)=1$. 
Let $T\in\End(\bS_{A})$ be an ergodic continuous endomorphism of $\bS_A$ $($see Lemma $\ref{l:erg}$ above$)$. Then the number $\abs{F_k(T)}$ of points of period $k\geq 1$ of $T$ is finite and
$$
\abs{F_k(T)}=\prod_{\fp\,\in\,\cS_\la\cup\, \cS_{\infty}}\abs{\xi^k-1}_{\fp},\quad \xi=\imath(T)\in\cO_{K,{\la}}.
$$
\end{prop}

\begin{proof} We adapt \cite[Lemma 5.1]{cew} and \cite[Lemma 5.2]{cew} to our case. 
By Proposition \ref{pr:charsol1}, Theorem \ref{th:onto} and \eqref{eq:funddom},
the fundamental domain 
of the action of $\Ga$ on $\bA_{K,\la}$ applied to $\widehat{G_{A^t}}$, has the form 
$$
\cF=[0,1)^n\times\prod_{p\in\cP}\prod_{\fp\in\cS_{\la},\,\fp\vert p}\bZ_p[{\bf w}],
$$
where ${\bf w}\in(\cO_K)^n$ is an eigenvector of $A^t$ corresponding to an eigenvalue $\la$.  
Let $T\in\End(G_{A^t})$ be the induced dual of $T$ (see the paragraph above the statement of the proposition).
By Proposition \ref{pr:end},
$T({\bf w})=\xi{\bf w}$ for some $\xi\in\cO_{K,{\la}}$. 
Thus, the induced action of $T$ on 
$\bA_{K,\la}$ 
is $T(\eta)=\xi\eta$, $\eta\in \bA_{K,\la}$. 
Since $\mu(\cF)\ne 0$, by \cite[Lemma 5.1]{cew}, we have
\begin{eqnarray*}
&& \abs{F_k(T)}=\abs{\ker(T^k-\id)}=\mu((T^k-\id)\cF)/\mu(\cF)=\prod_{\fp\,\in\,\cS_\la\cup\, \cS_{\infty}}\abs{\xi^k-1}_{\fp}.
\end{eqnarray*}
\end{proof}

\begin{rem}
Our formula for $\abs{F_k(T)}$ in Proposition \ref{pr:period} is consistent with earlier results from \cite{hl} and \cite{miles}. Recall that $G_A$ is a subgroup of $\bQ^n$ of rank $n$. In \cite{hl}, the authors provide a formula for $\abs{F_k(T)}$, when $T$ is a continuous endomorphism of the dual group $\widehat{G}$ of a subgroup $G$ of $\bQ^2$ of rank $2$. In \cite{miles}, the author presents a more general formula for when $\widehat{G}$ is a finite-dimensional compact abelian group. Our approach differs from both works. In particular, we describe $\bS_A$ using an ad\`ele ring. In \cite{miles}, the formula involves several global fields $K_1, \dots, K_n$, sets of finite places $\cP_i$ in $K_i$, and $\xi_i \in K_i$ \cite[Theorem 1.1]{miles}. In our case of $\widehat{G}=\bS_A$, under the assumptions in Proposition \ref{pr:period}, the number of global fields $n$ is $1$, with $K_1 = \bQ(\lambda)$ and $\cP_1$ consisting of prime ideals of the ring of integers of $K_1$ that do not divide $\lambda$. The correspondence between \cite[Theorem 1.1]{miles} and Proposition \ref{pr:period}  is established through the product formula.
\end{rem}

\begin{example}
In this example, we demonstrate how Proposition \ref{pr:period} can be used to count periodic points of a toroidal solenoid endomorphism when $n>2$. Let $n=3$, let $A\in\M_3(\bZ)$ be a matrix with characteristic polynomial  $h_A=x^3 - x^2 + 2x - 6\in\bZ[x]$. 
One can check that $h_A$ is irreducible, $\det A=6$, $\cP=\cP'=\{2,3\}$, so that the hypotheses on $A$ in Proposition \ref{pr:period} hold. Let $\la$ be a root of $h_A$, $K=\bQ(\la)$. It is known that $K$ is not Galois over $\bQ$, $\cO_K$ is a principal ideal domain, and $\cO_K=\bZ[\la]$ \cite[\href{https://www.lmfdb.org/EllipticCurve/Q/11.a2}{Number field 3.1.808.1}]{db}. Also, the ideals of $\cO_K$ generated by $2$, $3$ have the following prime decompositions: $2\cO_K=\fp_1^2\fp_2$ and $3\cO_K=\fq_1\fq_2$, where $\fp_1,\fp_2,\fq_1,\fq_2$ are prime ideals of $\cO_K$. There is a choice of $\la$ such that $\la\cO_K=\fp_1\fq_2$, $\fp_1=(2-\la)\cO_K$, $\fp_2=(5\la^2+4\la+17)\cO_K$, $\fq_1=(\la^2-\la-1)\cO_K$, $\fq_2=(\la^2+\la+3)\cO_K$ \cite{sage}. By Corollary \ref{cor:integ}, $\End(G_{A^t})\cong\bZ[\la^{\pm 1}]$. We pick an arbitrary element $\xi=(2+\la-\la^2)\la^{-3}\in\bZ[\la^{\pm 1}]$. It defines an endomorphism $\widehat{T}=(2I+A^t-(A^t)^2)(A^t)^{-3}$ 
of $G_{A^t}$, and its dual $T$ is a continuous ergodic endomorphism of $\bS_A$. One can compute $(\xi-1)\cO_K=\fp_1^{-2}\fq_2^{-3}\fa^2$, where 
$\fa$ is a prime ideal of $\cO_K$ above $13$ with the norm $N(\fa)=13$. Since $\cS_{\la}=\{\fp_1,\fq_2\}$, 
by Proposition \ref{pr:period}, the number of fixed points of $T$ is given by
$$
\abs{F_1(T)}=\prod_{\fp\,\in\,\cS_\la\cup\, \cS_{\infty}}\abs{\xi-1}_{\fp}=N(\fa)^2=169.
$$
Similarly, $(\xi^2-1)\cO_K=\fp_1^{-4}\fq_2^{-6}\fa^2\fb$, where 
$\fb$ is a prime ideal of $\cO_K$ above $229$ with the norm $N(\fb)=-229$. By Proposition \ref{pr:period}, the number of points of period $2$ of $T$ is given by
$$
\abs{F_2(T)}=\prod_{\fp\,\in\,\cS_\la\cup\, \cS_{\infty}}\abs{\xi^2-1}_{\fp}=\abs{N(\fa)^2N(\fb)}=169\times 229.
$$
\end{example}


\begin{cor}
Assume $A\in\M_n(\bZ)$ is non-singular with an irreducible characteristic polynomial 
and 
$\cP'\ne\emptyset$. 
Assume, in addition, that there exists a prime $p\in\bN$ with 
$(n,t_p)=1$. Let $T\in\End(\bS_{A})$, $\xi=\imath(T)\in\cO_{K,{\la}}$ $($see Proposition $\ref{pr:end})$. Assume 
$\xi$ is not a root of unity. Then the growth rate of the number of periodic points exists and is given by
$$
p^{+}(\xi)=p^{-}(\xi)=h_{top}(\xi).
$$
Here,
$$
p^{+}(\xi)=\lim\sup_{k\to\infty}\frac{1}{k}\log\abs{F_k(\xi)},\quad\quad  p^{-}(\xi)=\lim\inf_{k\to\infty}\frac{1}{k}\log\abs{F_k(\xi)},
$$
and $h_{top}(\xi)$ is the topological entropy of $\xi$.
\end{cor}
\begin{proof}
Follows from the proof of \cite[Theorem 6.1]{cew} and Proposition \ref{pr:period}.
\end{proof}

\section{Endomorphisms of $\bZ^n$-odometers} $\bZ^n$-odometer is a dynamical system 
consisting of a topological space $X$ and an action of the group $\bZ^n$ on $X$ (by homeomorphisms). 
Consider a decreasing sequence of finite-index subgroups of $\bZ^n$
$$
G=\bZ^n\supseteq G_1\supseteq G_2\supseteq\cdots
$$
and the natural maps $\pi_i:G/G_{i+1}\rar G/G_{i}$, $i\in\bN$. The associated $\bZ^n$-odometer is  the inverse limit 
\bbe\label{eq:xg}
X=\lim_{\longleftarrow} \left(G/G_{i}\right)
\ee
together with the natural action of $\bZ^n$. 
 For the sequence 
$$
G_i=\left.\left\{A^{i}{\bf x}\,\right\vert\, {\bf x}\in\bZ^n\right\},\quad i\in\bN,
$$
denote by $X_A$ the corresponding odometer. 
In \cite{cp}, the authors study the linear representation group of 
$X_A$ denoted by $\vec{N}(X_A)$. By \cite[Lemma 2.6]{cp}, $\vec{N}(X_A)$ consists of $T\in\GL_n(\bZ)$ ({\it i.e.}, $T\in\M_n(\bZ)$ with $\det T=\pm 1$) such that for any $m\in\bN\cup\{0\}$ there exists $k_m\in\bN\cup\{0\}$ with
\bbe\label{eq:linrepgr}
A^{-m}TA^{k_m}\in\M_n(\bZ).
\ee
By taking the transpose of the condition, one can see that it is
 equivalent to the condition that $T^t$ defines an endomorphism of $G_{A^t}$. 
\begin{lem}\label{lem:nvec}
$T\in\vec{N}(X_A)$ if and only if $T^t\in\End(G_{A^t})\cap\GL_n(\bZ)$.
\end{lem} 
 
 Therefore, our results 
can be applied to $T^t$. In particular, we can provide an alternate proof of \cite[Theorem 3.3]{cp} for the case when $n=2$ and also generalize some parts of it to an arbitrary $n$. 

\begin{lem}\label{lem:glnz}
Let $A\in\M_n(\bZ)$ be non-singular, let $F\subset \overline{\bQ}$ be any finite  extension of $\bQ$ that contains all the eigenvalues of $A$, and let $\cP'(A)\ne\emptyset$. 
Let $T\in\GL_n(\bZ)$. Then $T(G_A)\subseteq G_A$ if and only if 
\begin{eqnarray*}
&& T(\cX_{A,\fp})\subseteq \cX_{A,\fp} 
\end{eqnarray*}
for any $p\in\cP'$ and a prime ideal $\fp$ of $\cO_F$ above $p$.
\end{lem}
\begin{proof} Follows from Theorem \ref{th:prelim0}, since \eqref{eq:apcond} holds for any matrix $T\in\M_n(\bZ)$.
\end{proof}

\begin{cor}
Let $A\in\M_n(\bZ)$ be non-singular, let $F\subset \overline{\bQ}$ be any finite  extension of $\bQ$ that contains all the eigenvalues of $A$, and let $\cP'(A)\ne\emptyset$. 
Let $T\in\GL_n(\bZ)$. Then $T\in\vec{N}(X_A)$ if and only if 
\begin{eqnarray*}
&& T(\cX_{A^t,\fp})\subseteq \cX_{A^t,\fp} 
\end{eqnarray*}
for any $p\in\cP'$ and a prime ideal $\fp$ of $\cO_F$ above $p$.
\end{cor}
\begin{proof}
Follows from Lemma \ref{lem:nvec} and Lemma \ref{lem:glnz}.
\end{proof}

\begin{prop}\label{prop:mine}
Let $A\in\M_n(\bZ)$ be non-singular. 
\begin{enumerate}[$(1)$]
\item If $\cP(A)=\emptyset$ or $\cP'(A)=\emptyset$,  
then $\vec{N}(X_A)=\GL_n(\bZ)$. \\

\item Assume $\cP'(A)\ne\emptyset$. If the characteristic polynomial $h_A\in\bZ[x]$ is irreducible and there exists a prime $p\in\bN$ with 
$(n,t_p)=1$, then $\vec{N}(X_A)$ is the centralizer of $A$ in 
$\GL_n(\bZ)$. Moreover, $\vec{N}(X_A)$ is isomorphic to a subgroup of the group of units in the ring of integers of a number field $K=\bQ(\la)$ generated by an eigenvalue $\la\in\overline{\bQ}$ of $A$.
Therefore, as such, $\vec{N}(X_A)$ is a finitely generated abelian group and when $n=2$ and $h_A$ has no real roots, then
$\vec{N}(X_A)$ is finite.
%
\end{enumerate}
\end{prop}

Here, the condition $\cP(A)=\emptyset$ is equivalent to $\det A=\pm 1$. Recall that  
$$
\cP'(A)=\left\{p\in\cP(A),\,\,h_A\not\equiv x^n\,(\text{mod }p)\right\}.
$$
Thus, the condition $\cP'(A)=\emptyset$ in the case when $n=2$ is equivalent to the fact that every prime dividing $\det A$ also divides $\tr A$, $\operatorname{rad}(\det A)$ divides $\operatorname{trace} A$ in the notation of \cite{cp}. 
In the case when $n=2$, the conditions that the characteristic polynomial $h_A\in\bZ[x]$ is irreducible is equivalent to the fact that $A$ has no rational eigenvalues and the condition that there exists a prime $p\in\bN$ with 
$(n,t_p)=1$ holds automatically when $\cP'(A)\ne\emptyset$. 

\begin{proof}[Proof of Proposition $\ref{prop:mine}$]
Statement (1) follows from Lemma \ref{l:easy}. Statement (2) follows from Proposition $\ref{pr:end}$, since
by above, for any $T\in \vec{N}(X_A)$, $T^t\in\End(A^t)$. Moreover, since $T\in\GL_n(\bZ)$ by the definition of 
$\vec{N}(X_A)$, in the notation of the proof of Proposition $\ref{pr:end}$, $x=\imath(A^t)(T^t)$ is a unit in the ring of integers $\cO_K$ of $K$, {\it i.e.}, $x\in\cO_K^{\times}$. 
It is well-known that the group of units $\cO_K^{\times}$ of $\cO_K$ is finitely generated and when $n=2$ and $h_A$ has no real roots, then
$\cO_K^{\times}$ is finite.
\end{proof}

Recall that if $n=2$, $A\not\in\GL_2(\bZ)$, and $\cP'(A)\ne\emptyset$, there are three cases distinguished in \cite{s1}: \\

(a) $h_A\in\bZ[x]$ is irreducible (equivalently, $A$ has no rational eigenvalues), \\

(b) $h_A$ is reducible  (equivalently, $A$ has eigenvalues $\la_1,\la_2\in\bZ$), $\rad(\la_1)$ does not divide $\rad(\la_2)$, and $\rad(\la_2)$ does not divide $\rad(\la_1)$, \\

(c) $h_A$ is reducible and every prime 
$p\in\bN$ dividing one eigenvalue, divides the other, {\it e.g.}, $\rad(\la_1)$ divides $\rad(\la_2)$ (denoted by 
$\rad(\la_1)\,\vert \rad(\la_2)$). \\

\noindent Case (a) is covered by 
Proposition \ref{prop:mine} (2). The remaining two cases are covered in the next proposition ({\it c.f.}, \cite[Theorem 3.3 (b)]{cp}). 

\begin{prop}\label{prop:mine2} Let $A\in\M_2(\bZ)$ be non-singular, $\cP(A)\ne\emptyset$, and $\cP'(A)\ne\emptyset$.
Assume $h_A$ is reducible. Then  $A$ has distinct eigenvalues $\la_1,\la_2\in\bZ$.

$(1) $ Assume that $\rad(\la_1)$ does not divide $\rad(\la_2)$, and $\rad(\la_2)$ does not divide $\rad(\la_1)$.
\begin{itemize}
\item If there exists a matrix $M\in\M_2(\bZ)$ diagonalizing $A^t$ with $\det M$ dividing $2$, then  $\vec{N}(X_A)\cong\bZ/2\bZ\times\bZ/2\bZ$. 
\item Otherwise, $\vec{N}(X_A)=\{\pm \id\}$.
\end{itemize}

$(2)$ Assume $\rad(\la_1)\,\vert \rad(\la_2)$ or $\rad(\la_2)\,\vert \rad(\la_1)$, then $\vec{N}(X_A)$ is isomorphic to 
the group of lower-triangular matrices in $\GL_2(\bZ)$.
\end{prop}

\begin{proof} Recall that 
$T\in\vec{N}(X_A)$ if and only if $T^t\in\End(G_{A^t})\cap\GL_n(\bZ)$ by Lemma \ref{lem:nvec}.  
Then (1) follows from Theorem \ref{th:2dimredb} applied to $A^t$ and $T^t$. Indeed, for $T^t$ given by \eqref{eq:T}, we have that $T\in\GL_2(\bZ)$ if and only if 
$x_1=\pm 1$, $x_2=\pm 1$. Moreover, one can check that $T^t$ has integer coefficients if and only if $v\in\bZ$ divides 
$x_1-x_2$ in $\bZ$. 
Since $x_1-x_2=0$ or $x_1-x_2=\pm 2$, if $v$ divides $2$, then 
$\vec{N}(X_A)\cong\bZ/2\bZ\times\bZ/2\bZ$ and otherwise, $\vec{N}(X_A)=\{\pm \id\}$. 

\sbr

Similarly, 
(2) follows from Theorem \ref{th:2dimredc} applied to $A^t$ and $T^t$.
\end{proof}

\begin{rem}
In \cite{cp}, the authors conclude that $\vec{N}(X_A)$ is computable when $n=2$ and pose the question whether the group $\vec{N}(X_A)$ is computable when $n>2$. In \cite[p. 750, Main Algorithm 2]{ehb}, the authors show that the centralizer of an element in $\GL_n(\bZ)$ is computable.
Thus, if assumptions in $(1)$, $(2)$ of Proposition \ref{prop:mine} hold (the assumptions are computable), then 
$\vec{N}(X_A)$ is also computable by Proposition \ref{prop:mine} and \cite[p. 750, Main Algorithm 2]{ehb}.
%
%
\end{rem}

\begin{example}\label{ex:4}\cite[Example 10]{s1}. Let $A\in\M_n(\bZ)$ be non-singular with an irreducible characteristic polynomial  $h_A\in\bZ[x]$. By Proposition \ref{prop:mine} (2), if there exists $t_p$ with $(n,t_p)=1$, then $\vec{N}(X_A)$ is the centralizer of $A$ in 
$\GL_n(\bZ)$ and $\vec{N}(X_A)$ is abelian. 
In this example, we show that this is not always the case. 
Here $n=4$ and $t_p=2$, so the condition $(n,t_p)=1$ in Proposition \ref{prop:mine} (2) does not hold. 

\sbr

Let $h(x)=x^4-2x^3+21x^2-20x+5$, irreducible over $\bQ$, and let $\la\in\overline{\bQ}$ be a root of $h$, $K=\bQ(\la)$. 
By \cite{db}, $\cO_K=\bZ[\la]$, $K$ is Galois over $\bQ$, $\Gal(K/\bQ)\cong(\bZ/2\bZ)^2$.
Let 
$$
{\bf u}=\left(\begin{matrix}
1 & \la & \la^2 & \la^3
\end{matrix}
\right)^t
,\quad
A=\left(\begin{matrix}
0 & 1 & 0 & 0 \\
0 & 0 & 1 & 0 \\
0 & 0 & 0 & 1 \\
-5 & 20 & -21 & 2
\end{matrix}
\right),
$$
so that ${\bf u}$ is an eigenvector of $A$ corresponding to $\la$, and 
$A$ has characteristic polynomial $h_A(x)=h(x)=x^4-2x^3+21x^2-20x+5$. Also, $\det A=5$, $\cP(A)=\cP'(A)=\{5\}$, $t_5=2$.
By \cite{sage}, $(5)=\fp_1^2\fp_2^2$, where $\fp_1,\fp_2$ are prime ideals of 
$\bZ[\la]$, $\fp_1=(\la)$, and there exists $g\in \Gal(K/\bQ)$ of order 
$2$ such that 
$g(\fp_i)=\fp_i$, $i=1,2$. In the notation of \cite[Theorem 4.3]{s2}, 
$\cX_{A,\fp_1}=\Span_{K}({\bf u},g({\bf u}))$. We have that
$$
{\bf u}=\left(\begin{matrix}
1 & \la & \la^2 & \la^3
\end{matrix}\right)^t,\quad 
g({\bf u})=\left(\begin{matrix}
1 & g(\la) & g(\la^2) & g(\la^3)
\end{matrix}\right)^t,
$$
where 
\begin{eqnarray*}
g(\la)&=&-4\la^3+6\la^2-81\la+40,\\
g(\la^2)&=&-4\la^3+5\la^2-80\la+20,\\
g(\la^3)&=&75\la^3-114\la^2+1520\la-770.
\end{eqnarray*}
Then $g({\bf u})=L{\bf u}$, where $L\in\GL_4(\bZ)$ and 
$$
L=\left(\begin{matrix}
1 & 0 & 0 & 0 \\
40 & -81 & 6 & -4 \\
20 & -80 & 5 & -4 \\
-770 & 1520 & -114 & 75
\end{matrix}\right).
$$
Since $g({\bf u})$ is not a multiple of ${\bf u}$ (they are eigenvectors corresponding to two distinct eigenvalues $g(\la)$, $\la$, respectively), we see that ${\bf u}$ is not an eigenvector of $L$. Therefore, $A$ and $L$ do not commute. On the other hand, $L\in\vec{N}(X_A)$. Indeed, $\Gal(K/\bQ)$ acts transitively on the prime ideals $\fp_1,\fp_2$ above $5$, so there exists 
$g'\in\Gal(K/\bQ)$ such that $g'(\fp_1)=\fp_2$. 
By above, $L(\cX_{A,\fp_1})\subseteq \cX_{A,\fp_1}$ and applying $g'$, we get
$L(\cX_{A,\fp_2})\subseteq \cX_{A,\fp_2}$. By \cite[Theorem 4.3]{s2}, $L(G_A) \subseteq G_A$. On the other hand, for any 
unit $\xi\in\cO_K^{\times}$, let $T=T(\xi)$ be given by \eqref{eq:lm} and \eqref{eq:tx}, {\it i.e.},
\bbe\label{eq:tx1}
T=MXM^{-1},\quad X=\diag\left(\begin{matrix} \sg_1(\xi) & \sg_2(\xi) & \sg_3(\xi) &\sg_4(\xi)\end{matrix}\right),
\ee 
where $M$ is a matrix diagonalizing $A$ and $\Gal(K/\bQ)=\{\sg_1, \sg_2, \sg_3, \sg_4 \}$. Then, by construction, 
$T\in\GL_4(\bQ)$ and $\det T=\pm 1$. Since $\cO_K=\bZ[\la]$, $\xi=\sum_{i=0}^3a_i\la^i$ with $a_0,a_1,a_2,a_3\in\bZ$. 
This implies that $T=\sum_{i=0}^3a_iA^i$ and hence $T\in\GL_4(\bZ)$ and $T\in\vec{N}(X_A)$. Therefore, 
$\imath^{-1}(\cO_K^{\times})\subseteq \vec{N}(X_A)$. However, $L\not\in \imath^{-1}(\cO_K^{\times})$ and does not commute with the image. 
\end{example}

\end{document}